\documentclass[11pt,reqno]{amsart}

\usepackage{amssymb}
\usepackage{amscd}
\usepackage{amsfonts}
\usepackage{mathrsfs}
\usepackage{setspace}
\usepackage{version}
\usepackage{mathtools}
\usepackage{enumerate}
\usepackage[english]{babel}
\usepackage{xcolor}
\usepackage{graphicx}
\usepackage{hyperref}
\usepackage{xfrac}

\newtheorem{theorem}{Theorem}[section]
\newtheorem{lemma}[theorem]{Lemma}
\newtheorem{proposition}[theorem]{Proposition}
\newtheorem{corollary}[theorem]{Corollary}
\theoremstyle{definition}
\newtheorem{definition}[theorem]{Definition}

\newtheorem{example}[theorem]{Example}
\newtheorem{examples}[theorem]{Examples}

\theoremstyle{remark}
\newtheorem{remark}[theorem]{Remark}
\numberwithin{equation}{section}

\newcommand{\I}{I_{\min}}
\newcommand{\essinf}{{\rm ess.inf}}
\newcommand{\esssup}{{\rm ess.sup}}

\newcommand{\EE}{\mathcal{E}}

\newcommand{\MM}{\mathcal{M}}

\newcommand{\XX}{\mathcal{X}}

\newcommand{\field}[1]{\mathbb{#1}}

\newcommand{\R}{\field{R}}
\newcommand{\N}{\field{N}}

\begin{document}

	\title[]{Maxitive monetary risk measures: worst-case risk assessment and sharp large deviations}

		\author{Jos\'e M.~Zapata}
		\thanks{I thank two anonymous referees
for  careful reviews of the manuscript and comments which improved the presentation. This work is partially supported by Fundación Séneca - ACyT Región de Murcia project 21955/PI/22 and Ministerio de Ciencia e Innovación
of Spain in the project PID2022-137396NB-I00, funded by MICIU/AEI/10.13039/501100011033
and by `ERDF A way of making Europe'.
}
	\address{Universidad de Murcia.  Dpto. de Estadística e Investigación Operativa,  30100 Espinardo, Murcia, Spain}
	\email{jmzg1@um.es}

		\date{\today}

	\subjclass[2010]{}

	\begin{abstract} 
In decision making under uncertainty and risk, worst-case risk assessments are often conducted using maxitive monetary risk measures. In this article, we study maxitive monetary risk measures on the space 
$L^0$ of all random variables identified modulo almost sure equality. We prove that a monetary risk measure is maxitive and continuous from below if and only if it is a penalized maximum loss. Furthermore, we characterize the  maximum loss as the unique maxitive and law-invariant monetary risk measure. We apply the results to large deviation theory by providing
a general criterion to establish a sharp large deviation estimate for sequences of probability measures. We use these findings to provide a formula for the  asymptotics of the distortion-exponential insurance premium principle under risk pooling.

	\end{abstract}

	\maketitle
	
	\setcounter{tocdepth}{1}
\section{Introduction}

In decision making under uncertainty and risk, a monetary risk measure is a monotone and translation invariant function $\phi$ defined on a vector space  $E$ of functions, containing the constant functions, mapping to the extended real numbers. Here, each $f\in E$ represents the losses of a financial position (negative losses are interpreted as gains), and $\phi(f)$ indicates the additional capital required to make the position $f$ acceptable. For an introduction to risk measures, we refer to F\"{o}llmer and Schied \cite{follmer}.

Often, a monetary risk measure \(\phi\) takes the form of an expectation, where \(\phi(f)\) represents the weighted average of the losses associated with the financial position \(f\). However, risk management decisions frequently rely on assessments of worst-case scenarios rather than averages. These assessments are performed using maxitive monetary risk measures. A monetary risk measure \(\phi\) is said to be maxitive if \(\phi(f \vee g) \le \phi(f) \vee \phi(g)\) for any \(f, g \in E\), where \(f \vee g\) denotes the pointwise supremum of \(f\) and \(g\). { The use of maxitive monetary risk measures in risk analysis is explained in more detail in Section \ref{secmax} below. A  discussion on maxitive decision functions under uncertainty and risk can be also found in \cite{cattaneo}.} 
A basic example of maxitive monetary risk measure is  the \emph{maximum loss} ${\rm ML}(f):=\esssup_X(f)$, which gives the worst possible outcome. 
Another example of maxitive monetary risk measure is the \emph{penalized maximum loss}, which is defined as ${\rm ML}_I(f):=\esssup_X(f-I)$, where $I\colon X\to[0,\infty]$ is a measurable penalty function that penalizes certain outcomes, reflecting additional incomes,  capital reserves, or other loss mitigants  across different scenarios.

The class of maxitive monetary risk measures on the space $C_b(X)$  has been studied in \cite{kupper}. The aim of this paper is to drop the topological and continuity assumptions in \cite{kupper} and establish different representation results for maxitive risk measures on  the space $L^0$  of all real-valued measurable functions on a measurable space $(X,\mathcal{X})$, identified modulo $\mu$-almost sure equality for a given probability measure $\mu$ on $\XX$.  We establish that a monetary risk measure $\phi\colon L^0\to [-\infty,\infty]$ is maxitive and continuous from below on $L^\infty$ if and only if there exists a penalty $I$ such that $\phi(f)={\rm ML}_I(f)$ for all $f\in L^\infty$. 
Furthermore, under some regularity conditions on the underlying probability space $(X,\XX,\mu)$, we show that a monetary risk measure $\phi$ is maxitive and law-invariant if, and only if, $\phi(f)={\rm ML}(f)$. 
Moreover, if $\phi$ is maxitive on the entire space $L^0$, the aforementioned representations extend to functions $f \in L^0$ that are not necessarily bounded, provided that $\phi(tf)<\infty$ for some $t > 1$. 
Additionally, we complement these findings by establishing local conditions on $\phi$ that ensure the  representation as a penalized maximum loss.

Furthermore, by giving a different interpretation of our main results, we explore other  applications to large deviation theory and the limit theory of risk measures:
 
\subsection*{Sharp large deviation estimates} Originating from ruin theory, large deviation theory covers the asymptotic tail behaviour of sequences of probability distributions, see \cite{dembo} and references therein.  
Consider a sequence $(\nu_n)_{n\in\mathbb{N}}$ of probability measures on the measurable space $(X,\mathcal{X})$ such that $\nu_n \ll \mu$ for all $n\in\mathbb{N}$, for the reference probability measure $\mu$. The Varadhan functional $\phi\colon L^\infty\to\mathbb{R}$, defined by
\begin{equation}\label{eq:VaradhanFunctionalIntro}
\phi(f)=\lim_{n\to\infty} \tfrac{1}{n}\log\int_X e^{nf}\,{\rm d}\nu_n,
\end{equation}
shares the properties of a maxitive monetary risk measure\footnote{For simplicity in the exposition, we assume that the limit in $\phi(f)$ exists for all $f\in L^\infty.$}. In the special situation where $X$ is a regular topological space and $\mathcal{X}= \mathcal{B}(X)$ is the Borel $\sigma$-field on $X$, Varadhan's integral lemma asserts that the functional $\phi$ verifies the Laplace principle
\[
\phi(f)=\sup_{x\in X}(f(x)-I(x))\quad\text{ for all }f\in C_b(X)
\]
for a rate function $I\colon X\to[0,\infty]$, 
whenever the sequence $(\nu_n)_{n\in\mathbb{N}}$ satisfies the \emph{large deviation principle} (LDP) with rate  $I$, i.e.
\begin{equation}\label{eq:roughLDP}
-\inf_{x \in {\rm int}(A)} I(x) \le \liminf_{n\to\infty} \tfrac{1}{n} \log\nu_n(A) \le \limsup_{n\to\infty} \tfrac{1}{n} \log \nu_n(A) \le -\inf_{x \in {\rm cl}(A)} I(x)
\end{equation}
for all $A\in\mathcal{B}(X)$.\footnote{Here, ${\rm cl}(A)$ and ${\rm int}(A)$ respectively denote the closure and interior of $A$.} 

In the general case, where $(X,\mathcal{X})$ is an arbitrary measurable space, our representation result yields that the Varadhan functional adheres to the following form of the Laplace principle
\begin{equation}
\label{eq:essLP}
\phi(f)=\esssup_X(f-I)\quad\text{ for all }f\in L^\infty,
\end{equation}
where $I\colon X\to[0,\infty]$ is a measurable function, if and only if the Varadhan functional $\phi$ is continuous from below. Thus, our representation result turns out to be a non-topological extension of Varadhan's integral lemma. Furthermore, the Laplace principle \eqref{eq:essLP} holds if and only if the following sharp form of the LDP is satisfied
\begin{equation}
\label{eq:essLDP}
\lim_{n\to\infty}\tfrac{1}{n}\log\nu_n(A)=-\essinf_A I\quad\text{ for all }A\in\mathcal{X}.
\end{equation}  
The sharp LDP \eqref{eq:essLDP} was first introduced by Barbe and Broniatowski~\cite{barbe}, where it is proven that if $\nu_n$ are the laws of the sample means of a sequence of $\mathbb{R}^d$-valued i.i.d. random variables, then 
the sharp LDP \eqref{eq:essLDP} holds under some regularity conditions. 

We point out that while extensive research has been conducted on the rough LDP \eqref{eq:roughLDP} (see, e.g., \cite{dembo} and references therein), to the best of our knowledge, the sharp LDP \eqref{eq:essLDP} has not been studied beyond the specific case considered in \cite{barbe}.
 
As pointed out by Barbe and Broniatowski~\cite{barbe}, the rough LDP \eqref{eq:roughLDP} often yields very poor estimates compared to the sharp LDP \eqref{eq:essLDP}. For instance, if $A$ has empty interior and is dense (e.g., $A=(\mathbb{R}\setminus\mathbb{Q})^d$), the the left-hand side of \eqref{eq:roughLDP} is  $-\infty$ and the right-hand side of \eqref{eq:roughLDP} is  $0$. 
If $A$ is closed and $\inf_{x\in A}I=I(x_0)$  for a unique $x_0\in A$ which is 
isolated, then  $\inf_{x\in {\rm int}(A)}I(x)<\inf_{x\in A}I(x)$. 
Less trivial examples and discussions can be found in~\cite{barbe,slaby}.  

We use our representation result to provide a criterion for the sharp LDP \eqref{eq:essLDP} applicable to sequences of probability measures on a general measurable space $(X,\mathcal{X})$. Specifically, we establish that the sequence $(\nu_n)_{n\in\mathbb{N}}$ satisfies the sharp LDP \eqref{eq:essLDP} for a measurable function $I\colon X\to[0,\infty]$ whenever the effective domain $\{I<\infty\}$ has a measurable covering $\{I<\infty\}=\cup_{k\in\mathbb{N}} A_k$ such that the Radon-Nikodym derivatives $\tfrac{{\rm d}\nu_n}{{\rm d}\mu}$ satisfy 
\[
\tfrac{1}{n}\log \tfrac{{\rm d}\nu_n}{{\rm d}\mu} \to -I\quad\text{as } n\to\infty\quad\text{a.s. uniformly on }A_k
\]
for all $k\in\mathbb{N}$, and $\limsup_{n\to\infty}\tfrac{1}{n}\log \nu_n(A_k^c)\to -\infty$ as $k\to\infty$. 
We show that this criterion in fact covers the sharp LDP provided in \cite{barbe} for the sample means of a sequence of i.i.d.~random vectors.

\subsection*{A limit result for distorted expectations} 
Our theoretical framework covers  other monetary risk measures beyond the Varadhan functional. Specifically, for a concave distortion \(g \colon [0,1] \to [0,1]\), we define the distorted Varadhan functional as
\[
\phi_g(f) = \lim_{n \to \infty} \tfrac{1}{n} \log \mathcal{E}_{\nu_n, g}(e^{n f})\quad\mbox{ for all }f\in L^\infty,
\]
where \(\mathcal{E}_{\nu_n, g}(f) = \int_0^\infty g \circ \nu_n(f > x) \, \mathrm{d}x\) is the distorted expectation associated with \(g\). This functional retains the properties of a maxitive monetary risk measure.
Moreover, we establish that if the sequence \((\nu_n)_{n \in \mathbb{N}}\) satisfies the sharp LDP with rate $I$, then
\begin{equation}\label{eq:distortedLP}
\phi_g(f) = \esssup_X(f - \tfrac{1}{p} I),
\end{equation}
where \(p\) is the vanishing order of the distortion \(g\), i.e., the limit \(\lim_{x \downarrow 0} \tfrac{g(x)}{x^{1/p}}\) is finite and non-zero.

Tsanakas~\cite{tsanakas} introduced the \emph{distortion-exponential insurance premium principle}, a generalization of the exponential insurance premium principle, where the usual expectation is replaced by a distorted expectation.\footnote{In the exponential premium principle the premium to be paid for a claim $\xi$ is given by $\Pi_{\gamma}(\xi)=\tfrac{1}{\gamma}\log\mathbb{E}(e^{\gamma \xi})$ where $\gamma>0$ is a risk aversion parameter. 
In the distortion-exponential premium principle the premium to be paid for the claim $\xi$ is given by $\Pi_{g,\gamma}(\xi)=\tfrac{1}{\gamma}\log\mathcal{E}_{\mu,g}(e^{\gamma \xi})$. } 
Using \eqref{eq:distortedLP}, we derive the asymptotics of the distortion-exponential insurance premium principle under risk pooling as the number of claims in an insurance portfolio tends to infinity. Specifically, if \(\Pi_{g,\gamma}(\xi_1 + \xi_2 + \cdots + \xi_n)\) is the total premium for a homogeneous portfolio of random variables \(\xi_1, \xi_2, \dots, \xi_n\), assuming distortion \(g\) and risk aversion parameter \(\gamma > 0\), we prove that the  premium \(\pi_n\) per contract\footnote{Since the portfolio is homogeneous, all insured individuals pay the same premium, which is calculated as $\pi_n=\tfrac{1}{n}\Pi_{g,\gamma}(\xi_1 + \xi_2 + \cdots + \xi_n)$.} satisfies
\[
\lim_{n \to \infty} \pi_n = \esssup_{x \in \mathbb{R}} \left( x - \tfrac{1}{p \gamma} I \right),
\]
whenever \(g\) has vanishing order \(p\) and the laws \(\nu_n\) of the averages \(\tfrac{1}{n} (\xi_1 + \cdots + \xi_n)\) satisfy the sharp LDP with rate  \(I\).
 In the special case where the claims \(\xi_1, \xi_2, \dots\) are i.i.d., the individual premium \(\pi_n\) satisfies
\[
\lim_{n \to \infty} \pi_n = \Pi_{p\gamma}(\xi),
\]
where \(\Pi_{p\gamma}\) is the premium obtained using the standard exponential premium principle with risk aversion parameter \(p\gamma\). That is, the distortion-exponential premium principle converges to the usual exponential premium principle where the risk aversion parameter $p\gamma$ is given by the vanishing order of the distortion $g$.

The remainder of the paper is organized as follows.
Section~\ref{sec:preliminaries} introduces the setting, basic concepts, and definitions relevant to our study.
{ In Section~\ref{secmax}, we present the concept of maxitive risk measures and interpret them as tools for worst-case risk evaluation.} 
Section~\ref{sec:mainResults} states the main representation results.
Section~\ref{sec:applications} is devoted to applications of these results, including sharp estimates for sequences of probability measures, limit results for sequences of distorted expectations, and the asymptotic behavior of risk pooling under the distortion-exponential premium principle.
We conclude with an appendix containing auxiliary results and the proofs of the main theorems.

\section{Preliminaries}\label{sec:preliminaries}

\subsection{Notation and Setting}

Throughout this paper, \((X, \mathcal{X})\) denotes a measurable space, and $\mu$ a fixed reference probability measure \(\mu\) on \(\mathcal{X}\). 
We denote by \(L^0\) the space of all (equivalence classes modulo \(\mu\)-almost sure equality of) real-valued measurable functions on \(X\), and by \(L^p := L^p(X, \mathcal{X}, \mu)\), for \(p \in [1,\infty]\), the standard Lebesgue subspaces of \(L^0\). Given two functions  $f, g \in L^0$, we denote by $f \vee g$ (resp., $f \wedge g$) the  pointwise maximum (resp., minimum) of $f$ and $g$. 
In this paper, we deal with two types of essential suprema. The \emph{essential supremum} of a measurable function \(f\colon X \to [-\infty,\infty]\), denoted by \(\esssup_X f\), is defined as 
\[
\esssup_X f = \inf\{r \in \mathbb{R} \colon \mu(f \leq r) = 1\}.
\]
Additionally, the \emph{essential supremum} of a family \(\mathcal{F}\) of measurable functions from \(X\) to \([- \infty, \infty]\) is denoted by
\[
\esssup \mathcal{F} = \underset{f \in \mathcal{F}}{\esssup} f;
\]
we refer to the Section~\ref{sec:A1} in the Appendix for the definition and some properties. We denote by \(\mathcal{M}_1\) the set of all probability measures \(\nu\) on \((X, \mathcal{X})\) such that \(\nu \ll \mu\) (i.e., \(\nu\) is absolutely continuous with respect to \(\mu\)).

\subsection{Monetary Risk Measures}
In this subsection, let \(E\) be a vector sublattice of \(L^0\), that is, \(E\) is a vector subspace of \(L^0\) such that for any \(f, g \in E\), both \(f \vee g\) and \(f \wedge g\) are also elements of \(E\). Additionally, we assume that \(L^\infty \subset E\). 

\begin{remark} Typical examples of vector sublattices of \( L^0 \) that contain \( L^\infty \) are the standard Lebesgue spaces \( L^p \) for \( p \geq 1 \), see e.g.~\cite{aliprantis}. \end{remark}

\begin{definition}
A \emph{monetary risk measure}\footnote{The definition provided here aligns with actuarial science literature, which differs slightly from that used in the mathematical finance literature (see, e.g., \cite{follmer}), where \(\rho\) is a monetary risk measure if \(\phi(f) = \rho(-f)\) satisfies (N), (M), and (T) as defined above.} is a function \(\phi\colon E \to [-\infty,\infty]\) that satisfies the following axioms:
\begin{itemize}
\item[(N)] \emph{Normalization}: \(\phi(0) = 0\),
\item[(M)] \emph{Monotonicity}: If \(f \leq g\)~a.s., then \(\phi(f) \leq \phi(g)\),
\item[(T)] \emph{Translation invariance}: \(\phi(f + c) = \phi(f) + c\) for all \(c \in \mathbb{R}\).
\end{itemize}
\end{definition}
Let \(\phi\colon E \to [-\infty,\infty]\) be a monetary risk measure.
\begin{itemize}
\item We say that \(\phi\) is \emph{continuous from below} if
\[
f_n \uparrow f \text{ a.s. in } E \quad \text{implies} \quad \phi(f_n) \uparrow \phi(f).
\]
\item We say that \(\phi\) is \emph{convex} if 
\[
\phi(t f + (1-t) g) \leq t \phi(f) + (1-t) \phi(g) \quad \text{for all } f, g \in E \text{ and } t \in (0,1).
\]
\item \(\phi\) is \emph{law-invariant} if \(\phi(f) = \phi(g)\) whenever \(f\) and \(g\) have the same distribution under \(\mu\) (i.e., $\mu(f\le x)=\mu(g\le x)$ for all $x\in\R$).
\end{itemize}

In risk analysis, each element \( f \in E \) represents the random losses of a financial position, where negative values correspond to gains. The \emph{acceptance set} of a risk measure \( \phi \) is defined as
\[
\mathcal{A} := \{f \in E \colon \phi(f) \leq 0\}.
\]
This set \( \mathcal{A} \) is interpreted as the collection of all positions that are considered  acceptable, in the sense that they do not require any additional capital. The monetary risk measure \( \phi \) is uniquely determined by its acceptance set through the relation
\[
\phi(f) = \inf\{c \in \mathbb{R} \colon f + c \in \mathcal{A}\},
\]
where the infimum is attained as a minimum whenever \( \phi(f) \) is finite. 
In other words, \( \phi(f) \) represents the minimal amount of capital that must be added to the position \( f \) in order to render it acceptable.

Every monetary risk measure \(\phi\) on \(L^\infty\) is necessarily real-valued. The \emph{convex conjugate} of a convex risk measure \(\phi\colon L^\infty \to \mathbb{R}\) is defined as
\[
\phi^\ast(\nu) := \sup_{f \in L^\infty} \left\{\int_X f \,{\rm d}\nu - \phi(f)\right\} = \sup_{f \in \mathcal{A}} \int_X f \, {\rm d}\nu,
\]
where the last equality follows from (T). The property of being continuous from below for a convex risk measure \(\phi\colon L^\infty \to \mathbb{R}\) is equivalent to lower semicontinuity in the weak-$\ast$ topology \(\sigma(L^\infty, L^1)\) and leads to the following representation
\[
\phi(f) = \sup_{\nu \in \mathcal{M}_1} \left\{\int_X f \, {\rm d}\nu - \phi^\ast(\nu)\right\} \quad \text{for all } f \in L^\infty;
\]
see~\cite[Theorem 4.33]{follmer}.\footnote{We use the terminology of \cite{follmer} with a sign change; what is termed continuity from above in \cite{follmer} corresponds to continuity from below in this paper.}

\section{Maxitive Risk Measures and Worst-Case Risk Assessment}\label{secmax}

Let \( E \) be a vector sublattice of \( L^0 \) that contains \( L^\infty \). The concept of a \emph{maxitive monetary risk measure}, originally studied in~\cite{kupper}, can be adapted to our setting as follows.

\begin{definition} 
A monetary risk measure $\phi\colon E\to [-\infty,\infty]$ is called \emph{maxitive} if it satisfies the inequality $\phi(f\vee g)\leq \phi(f)\vee \phi(g)$ for all $f,g\in E$. We say that $\phi$ is \emph{$\sigma$-maxitive} if, for every sequence $(f_n)_{n\in\mathbb{N}}$ in $E$ such that $\sup_{n\in\mathbb{N}} f_n\in E$, it holds that $\phi(\sup_{n\in\mathbb{N}} f_n)\leq \sup_{n\in\mathbb{N}} \phi(f_n)$. 
\end{definition}

\begin{remark}
By monotonicity, any maxitive monetary risk measure automatically satisfies $\phi(f \vee g) = \phi(f) \vee \phi(g)$ for all $f, g \in E$. 
Every \( \sigma \)-maxitive risk measure is in particular maxitive. Moreover, every maxitive monetary risk measure is convex; see~\cite[Proposition 2.1]{kupper}.
\end{remark}

A canonical example of a \( \sigma \)-maxitive monetary risk measure is the \emph{maximum loss} $\mathrm{ML}(f) := \esssup_X f.$ 
A penalized version is given by
 $\mathrm{ML}_I(f) := \esssup_X (f - I)$, 
where \( I\colon X \to [0, \infty] \) is a measurable function with \( \essinf_X I = 0 \).

A {maxitive} monetary risk measure captures a \emph{worst-case} evaluation of risk. This can be seen as follows. In Proposition~\ref{prop:max} below, we show that a monetary risk measure \( \phi \) is maxitive if and only if, for every financial position \( f \in E \) and every event \( A \in \mathcal{X} \), one has
\begin{equation}\label{eq:decomposition}
\phi(f) = \phi_A(f) \vee \phi_{A^c}(f),
\end{equation}
where \( \phi_A(f) \) denotes the \emph{localized risk evaluation} of \( f \) on \( A \), formally defined in~\eqref{eq:locRiskEv} below.

The decomposition~\eqref{eq:decomposition} implies that, for any set \( A \), the risk evaluation \( \phi(f) \) corresponds to the worst of its localized risk evaluations on \( A \) and \( A^c \). This shows that worst-case risk assessments are  characterized by maxitive monetary risk measures.

Another  feature of maxitive monetary risk measures is that, in contrast to expectations, which average losses across different states, a maxitive monetary risk measure does not allow large profits in some states to compensate for losses in others. To see this, suppose that \( \phi_A(f) \ge \phi_{A^c}(f) \) in~\eqref{eq:decomposition}. Then, for any \( g \in E \) that agrees with \( f \) a.s.~on \( A \), we still have
\[
\phi(f) = \phi_A(f) = \phi_A(g),
\]
regardless of the values that \( g \) takes on \( A^c \). In particular, this implies that  gains on \( A^c \) cannot offset losses  on \( A \).

As a result, maxitivity ensures that risk is not underestimated by averaging outcomes across different scenarios. From a financial perspective, maxitive monetary risk measures align well with a conservative approach to risk assessment, particularly in contexts where {robustness} and {prudence} are critical. This is especially relevant in settings such as {stress testing} or {regulatory capital requirements}, where the goal is to ensure resilience under {extreme adverse events}.
 However, such an approach may be too  pessimistic in  portfolio management,  as it neglects potential benefits of average performance. 

We now formalize this characterization.

\begin{proposition}\label{prop:max}
Let \( \phi\colon E \to [-\infty, \infty] \) be a monetary risk measure. The following are equivalent:
\begin{enumerate}
    \item \label{item:maxitive} \( \phi \) is maxitive.
    \item \label{item:decomp} For all \( f \in E \) and all \( A \in \mathcal{X} \),
    \[
    \phi(f) = \phi_A(f) \vee \phi_{A^c}(f),
    \]
    where
    \begin{equation}\label{eq:locRiskEv}
    \phi_A(f) := \inf\left\{ \phi(g) \colon g \in E,\, f = g \text{ a.s. on } A \right\}.
    \end{equation}
\end{enumerate}
\end{proposition}

\begin{remark}
In the definition of \( \phi_A(f) \), the infimum is taken over all positions \( g \in E \) that agree with \( f \) on \( A \), no matter what they do on \( A^c \). By allowing arbitrarily large gains on \( A^c \), we ignore the risk in that region. As a result, \( \phi_A(f) \) captures the localized risk evaluation of \( f \) on \( A \).
\end{remark}

\begin{proof}
{(1) \(\Rightarrow\) (2):}  
Let \( g, h \in E \) be such that \( g = f \) a.s. on \( A \) and \( h = f \) a.s. on \( A^c \). Since $\phi$ is maxitive and $f=(f\wedge g)\vee (f\wedge h)\in E$,\footnote{Recall that $E$ is a lattice.}
\[
\phi(f) = \phi(f \wedge g) \vee \phi(f \wedge h).
\]
Taking first the infimum over all $g\in E$ such that $g=f$~a.s. on $A$, and then the infimum over all $h\in E$ such that $h=f$~a.s. on $A$, we obtain $\phi(f)=\phi_A(f)\vee \phi_{A^c}(f)$.

{(2) \(\Rightarrow\) (1):}  
Let \( f, g \in E \), and define  \( A := \{f \ge g\} \). Then, using  (2) twice,
\begin{align*}
\phi(f \vee g) &= \phi_A(f \vee g) \vee \phi_{A^c}(f \vee g) \\
&= \phi_A(f) \vee \phi_{A^c}(g) \\
&\le (\phi_A(f) \vee \phi_A(g)) \vee (\phi_{A^c}(f) \vee \phi_{A^c}(g)) \\
&= (\phi_A(f) \vee \phi_{A^c}(f)) \vee (\phi_A(g) \vee \phi_{A^c}(g)) \\
&= \phi(f) \vee \phi(g).
\end{align*}
Thus, \( \phi \) is maxitive.
\end{proof}

\section{Main results}\label{sec:mainResults}

In this section, we present various representation results for maxitive monetary risk measures on \(L^0\). It is  not restrictive to consider the entire space \(L^0\) as the domain. In fact, given a maxitive monetary risk measure \(\phi\) on a sublattice \(E \subset L^0\), we can always consider its extension \(\overline{\phi}\colon L^0 \to [-\infty,\infty]\) defined by
\[
\overline{\phi}(f) := \inf_{g \in E \colon f \leq g} \phi(g),
\]
which is a maxitive monetary risk measure on \(L^0\). 
Throughout this section, we fix a monetary risk measure \(\phi\colon L^0 \to [-\infty,\infty]\) and define its associated   \emph{concentration} \(J\colon \mathcal{X} \to [-\infty,0]\) by
\begin{equation}
\label{eq:concentration}
J_A := \inf_{r < 0} \phi(r1_{A^c}) \quad \text{for each } A \in \mathcal{X}.
\end{equation}
Additionally, we introduce the \emph{minimal penalty}
\begin{equation}
\label{eq:minrate0}
\I := \esssup\{f - \phi(f) \colon f \in L^\infty\},
\end{equation}
and define the set
\begin{equation}\label{eq:unbounded}
L^\phi := \{f \in L^0 \colon \text{there exists } t > 1 \text{ such that } \phi(tf) < \infty\}.
\end{equation}

We now state the main results of this paper, deferring their proofs to Section~\ref{sec:A2} in the Appendix. 
The first main result of this section is presented below.

\begin{theorem}\label{thm:main1}
The following conditions are equivalent:
\begin{enumerate}
    \item \(\phi|_{L^\infty}\) is continuous from below and maxitive on \(L^\infty\).
    \item \(\phi|_{L^\infty}\) is \(\sigma\)-maxitive on \(L^\infty\).
    \item \(\phi|_{L^\infty}\) is convex and continuous from below on \(L^\infty\), and there exists a measurable function \(I\colon X \to [0,\infty]\) such that
    \[
    (\phi|_{L^\infty})^\ast(\nu) = \int_X I \, {\rm d}\nu \quad \text{for all } \nu \in \mathcal{M}_1.
    \]
    \item \(\phi|_{L^\infty}\) is maxitive on \(L^\infty\) and there exists a measurable function \(I\colon X \to [0,\infty]\) such that
    \begin{equation}\label{eq:concentrationRep}
    J_A = -\essinf_A I \quad \text{for all } A \in \mathcal{X}.
    \end{equation}
    \item There exists a measurable function \(I\colon X \to [0,\infty]\) such that
    \begin{equation}\label{eq:repLinfty}
    \phi(f) = {\rm ML}_I(f) \quad \text{for all } f \in L^\infty.
    \end{equation}
\end{enumerate}
If \(\phi\) is maxitive on $L^0$, then the conditions \textup{(1)}--\textup{(5)} are also equivalent to:
\begin{enumerate}\setcounter{enumi}{5}
    \item There exists a measurable function \(I\colon X \to [0,\infty]\) such that
    \begin{equation}\label{eq:repLphi}
    \phi(f) = {\rm ML}_I(f) \quad \text{for all } f \in L^\phi.
    \end{equation}
\end{enumerate}

Moreover, if any of the properties \textup{(3)}--\textup{(6)} is satisfied for a measurable function \(I\colon X \to [0,\infty]\), then necessarily \(I = \I\) a.s.
\end{theorem}

\begin{remark}
The equivalence \((1) \Leftrightarrow (5)\) establishes that a monetary risk measure on \(L^\infty\) is maxitive and continuous from below if and only if it can be represented as a penalized maximum loss. 
Furthermore, the implication \((1) \Rightarrow (6)\)  shows that, under the additional requirement of maxitivity on the entire domain $L^0$, the penalized maximum loss representation extends to functions that are not necessarily bounded. 
\end{remark}

\begin{remark}
Theorem~\ref{thm:main1} is  related to Theorem~2.1 in \cite{kupperzapata}, as well as its variants Theorems~3.1 and~4.1 in the same reference, and Theorem~4.7 in \cite{zapata}. All of these results establish conditions under which a monetary risk measure \(\phi\), defined on the space \(C_b(S)\) of bounded continuous functions over a topological space \(S\), admits a representation of the form
\[
\phi(f) = \sup_{x \in S} \big(f(x) - I(x)\big) \quad \text{for all } f \in C_b(S).
\]
The proofs of these results fundamentally rely on the topological structure of \(S\),  exploiting topological properties such as continuity, compactness, and different types of separation.

In contrast, Theorem~\ref{thm:main1} provides an analogous representation  \eqref{eq:repLinfty}   without invoking any topological assumptions. It only requires measurability, and is formulated in the general setting of a measurable space, thereby extending the scope of representation results beyond a  topological framework. 
\end{remark}

The following variant of Theorem \ref{thm:main1} addresses the law-invariant case.

\begin{theorem}\label{thm:main2}
The following conditions are equivalent:
\begin{enumerate}
    \item \( J_A = 0 \) whenever \(\mu(A) > 0\).
    \item \(\phi(f) = {\rm ML}(f)\) for all \(f \in L^\infty\).
\end{enumerate}
If \(\phi\) is maxitive on $L^0$, then the conditions \textup{(1)}--\textup{(2)} are also equivalent to:
\begin{enumerate}\setcounter{enumi}{2}
    \item \(\phi(f) = {\rm ML}(f)\) for all \(f \in L^\phi\).
\end{enumerate}
If \(\mu\) is atomless and \(L^2\) is separable, then conditions \textup{(1)}--\textup{(2)} are also equivalent to:
\begin{enumerate}\setcounter{enumi}{3}
    \item \(\phi|_{L^\infty}\) is maxitive and law-invariant on \(L^\infty\).
\end{enumerate}
\end{theorem}

\begin{remark}
Unlike condition \textup{(4)} in Theorem \ref{thm:main1}, the assumption that \(\phi\) is maxitive is not required for condition \textup{(1)} in Theorem \ref{thm:main2}.
\end{remark}

\begin{remark}\label{rem:lawinvariant}
Theorem \ref{thm:main2} establishes that any maxitive, law-invariant monetary risk measure on $L^\infty$ must necessarily be the maximum loss \({\rm ML}\), provided that the probability measure \(\mu\) is atomless and the space \(L^2\) is separable.
\end{remark}

\begin{remark}
We point out that Theorem \ref{thm:main2} provides simple  criteria to identify when a seemingly complex risk measure is, in fact, simply the maximum loss. This is further illustrated in \textup{(2)} and \textup{(3)} in Examples \ref{examples}.
\end{remark}

Our next main result reads as follows.

\begin{theorem}\label{thm:main3}
The following conditions are equivalent:
\begin{enumerate}
    \item \(\phi\) is maxitive and continuous from below on $L^0$.
    \item \(\phi\) is \(\sigma\)-maxitive on $L^0$.
    \item There exists a measurable function \(I \colon X \to [0, \infty]\) such that 
    \begin{equation}
    \phi(f) = {\rm ML}_I(f) \quad \text{for all } f \in L^0.
    \end{equation}
\end{enumerate}
In this case, \(I = \I\) a.s.
\end{theorem}

\begin{remark}
Theorem \ref{thm:main3} fully characterizes the class of $\sigma$-maxitive monetary risk measures on \(L^0\). In particular, we have that if in (6) of Theorem \ref{thm:main1} we strengthen the condition of maxitivity and assume that $\phi$ is
$\sigma$-maxitive, the penalized loss representation is valid on the whole domain $L^0$.
\end{remark}

We have shown that a maxitive monetary risk measure $\phi$ is a penalized maximum loss whenever $\phi$ is continuous from below. 
However, establishing the continuity from below for a monetary risk measure may   be challenging. The following result offers local conditions for determining when a maxitive risk measure can be represented as a penalized maximum loss.

\begin{theorem}\label{thm:main4}
Suppose that \(\phi\) is maxitive on $L^0$, and let \(I \colon X \to [0, \infty]\) be a measurable function. Assume there exists a countable family \((A_k)_{k \in \mathbb{N}} \subset \mathcal{X}\) satisfying the following conditions:
\begin{enumerate}
    \item \(\{I < \infty\} = \bigcup_{k \in \mathbb{N}} A_k\) a.s.
    \item \(J_{A_k^c} \to -\infty\) as \(k \to \infty\).
    \item For each \(k \in \mathbb{N}\), \(J_B = -\essinf_B I\) for all \(B \in \mathcal{X}\) with \(B \subset A_k\).
\end{enumerate}
Then, the following holds
\begin{equation}
\phi(f) = {\rm ML}_I(f) \quad \text{for all } f \in L^\phi.
\end{equation}
Moreover, in this case, \(I = \I\) a.s.
\end{theorem}

Below, we present some examples of maxitive risk measures to which the theorems in this section are applicable.

\begin{examples}\label{examples}
\begin{enumerate}

\item   
A function \(\mathcal{E} \colon L^0_+ \to [0, \infty]\) is called a \emph{sublinear expectation} if it satisfies the following properties:
\begin{enumerate}
    \item \(\mathcal{E}(a) = a\) for every constant \(a \geq 0\).
    \item \(\mathcal{E}(f) \leq \mathcal{E}(g)\) whenever \(f \leq g\) a.s.
    \item \(\mathcal{E}(f + g) \leq \mathcal{E}(f) + \mathcal{E}(g)\).
    \item \(\mathcal{E}(af) = a \mathcal{E}(f)\) for all constants \(a \geq 0\).
\end{enumerate}

Given a sequence \((\mathcal{E}_n)_{n \in \mathbb{N}}\) of sublinear expectations, the monetary risk measure \(\phi \colon L^0 \to [-\infty, \infty]\) defined by
\[
\phi(f) := \limsup_{n \to \infty} \tfrac{1}{n} \log \mathcal{E}_n(e^{nf}),
\]
is maxitive. To see this, consider \(f, g \in L^0\). Applying the \emph{principle of the largest term}\footnote{The \emph{principle of the largest term} is a result frequently used in large deviation theory. It states that for any two \([0, \infty]\)-valued sequences \((a_n)_{n \in \mathbb{N}}\) and \((b_n)_{n \in \mathbb{N}}\), we have \(\limsup_{n \to \infty} \tfrac{1}{n} \log(a_n + b_n) = \left(\limsup_{n \to \infty} \tfrac{1}{n} \log(a_n)\right) \vee \left(\limsup_{n \to \infty} \tfrac{1}{n} \log(b_n)\right)\), see~\cite[Lemma 1.2.15]{dembo}.}, we get
\begin{align*}
\phi(f \vee g) &\le \limsup_{n \to \infty} \tfrac{1}{n} \log \mathcal{E}_n(e^{n(f \vee g)}) \\
&\le \limsup_{n \to \infty} \tfrac{1}{n} \log \mathcal{E}_n(e^{nf} + e^{ng}) \\
&\le \limsup_{n \to \infty} \tfrac{1}{n} \log \left(\mathcal{E}_n(e^{nf}) + \mathcal{E}_n(e^{ng})\right) \\
&= \left(\limsup_{n \to \infty} \tfrac{1}{n} \log \mathcal{E}_n(e^{nf})\right) \vee \left(\limsup_{n \to \infty} \tfrac{1}{n} \log \mathcal{E}_n(e^{ng})\right) \\
&= \phi(f) \vee \phi(g).
\end{align*}

\item  
The two most popular monetary risk measures in financial practice are Value at Risk (VaR) and Average Value at Risk (AVaR).  
The monetary risk measure VaR at level $\alpha \in (0,1)$ is the functional ${\rm VaR}_{\mu,\alpha}\colon L^\infty \to \mathbb{R}$ defined by
\[
{\rm VaR}_{\mu,\alpha}(f) := q_{\mu,\alpha}(f) = \inf\left\{m \in \mathbb{R} \colon \mu(f \le m) \ge \alpha \right\},
\]
where $q_{\mu,\alpha}$ is the left $\alpha$-quantile of the distribution of $f$ under $\mu$. 
In addition, the monetary risk measure AVaR at level $\alpha \in (0,1)$ is the functional ${\rm AVaR}_{\mu,\alpha}\colon L^\infty \to \mathbb{R}$ defined by
\[
{\rm AVaR}_{\mu,\alpha}(f) := \frac{1}{1-\alpha} \int_{\alpha}^1 {\rm VaR}_{\mu,s}(f) \, {\rm d}s.
\]
We can use Theorem \ref{thm:main2} to study the asymptotic behavior of AVaR.  
Suppose that the sequences $(\nu_n)_{n\in\mathbb{N}} \subset \mathcal{M}_1$ and $(\alpha_n)_{n\in\mathbb{N}} \subset (0,1)$ satisfy
\begin{equation}\label{eq:asympCond} 
\limsup_{n\to\infty} \frac{\nu_n(A)}{1-\alpha_n} \ge 1 \quad \text{for all } A \in \mathcal{X} \text{ with } \mu(A) > 0.
\end{equation}
Then, we claim that
\begin{equation}\label{eq:AVAR}
\limsup_{n\to\infty} {\rm AVaR}_{\nu_n,\alpha_n}(f) = \operatorname{ess\,sup}_X f
\quad \text{for all } f \in L^\infty.
\end{equation}
Notice that the condition \eqref{eq:asympCond} is satisfied if $\nu_n = \mu$ for all $n \in \mathbb{N}$ and $\alpha_n \uparrow 1$.  

To show \eqref{eq:AVAR}, consider the monetary risk measure
\[
\phi(f) := \limsup_{n\to\infty} {\rm AVaR}_{\nu_n,\alpha_n}(f).
\] 
Fix $r < 0$ and $A \in \mathcal{X}$ with $\mu(A) > 0$. We have
\[
\nu_n(r 1_{A^c} \le m) = \left(1 - \nu_n(A)\right) 1_{[r,0)}(m) + 1_{[0,\infty)}(m).
\]
Then, for $s \in (0,1)$, it holds that
\[
{\rm VaR}_{\nu_n,s}(r 1_{A^c}) = r 1_{(0,1-\nu_n(A)]}(s).
\]
Thus,
\[
{\rm AVaR}_{\nu_n,\alpha_n}(r 1_{A^c}) = \frac{r}{1-\alpha_n} \int_{\alpha_n}^1  1_{(0,1-\nu_n(A)]}(s) \, {\rm d}s
= r \left|1 - \frac{\nu_n(A)}{1-\alpha_n}\right|^+.
\]
Therefore, we get
\[
\phi(r 1_{A^c}) = \limsup_{n\to\infty} {\rm AVaR}_{\nu_n,\alpha_n}(r 1_{A^c}) = r \liminf_{n\to\infty} \left|1 - \frac{\nu_n(A)}{1-\alpha_n}\right|^+ = 0,
\]
where the last equality follows from \eqref{eq:asympCond}.  
Consequently,  
\[
{J}_A := \inf_{r < 0} \phi(r1_{A^c}) = 0.
\]
Since the equality above holds for every $A \in \mathcal{X}$, 
by Theorem~\ref{thm:main2} it follows that
\[
\phi(f) = \operatorname{ess\,sup}_X f \quad \text{for all } f \in L^\infty,
\]
as required.

\item  A   distortion is a non-decreasing function $g\colon[0,1] \to [0,1]$ with $g(0) = 0$ and $g(1) = 1$. 
The \emph{distorted expectation} $\mathcal{E}_{\mu,g}\colon L^0_+ \to [0,\infty]$ associated with the   distortion $g$ is defined by
\[
\mathcal{E}_{\mu,g}(f) = \int_0^{\infty} g\circ\mu(f > x) {\rm d}x. 
\]
If $g$ is concave, then $\mathcal{E}_{\mu,g}$ is a sublinear expectation, see e.g.~\cite{wirch}.  
Suppose that $\mu$ is atomless and $L^2$ is separable. We claim that, for every sequence $(g_n)_{n\in\mathbb{N}}$ of concave distortions, it holds
\begin{equation}\label{eq:LPexpLawInv}
 \lim_{n\to\infty} \tfrac{1}{n} \log \mathcal{E}_{\mu,g_n}(e^{n f}) = \esssup_X f \quad \text{for all } f \in L^\infty.
\end{equation}
To prove this, we consider the monetary risk measure 
\[
\phi(f) := \limsup_{n\to\infty} \tfrac{1}{n} \log \mathcal{E}_{\mu,g_n}(e^{n f}),
\] 
which law invariant and maxitive (see Example (1) above). 
As a consequence of Theorem~\ref{thm:main3}, we have 
\[
\phi(f) = \esssup_X f \quad \text{for all } f \in L^\infty.
\] 
To obtain \eqref{eq:LPexpLawInv} we show that the limit superior in $\phi(f)$ is a limit for all $f \in L^\infty$. Suppose, by contradiction, that the limit in $\phi(f)$ does not exist for some $f \in L^\infty$. In that case, we can take a subsequence $n_1 < n_2 < \cdots$ in $\mathbb{N}$ such that
\begin{equation}\label{eq:subsequence}
 \lim_{k\to\infty} \tfrac{1}{n_k} \log \mathcal{E}_{\mu,g_{n_k}}(e^{n_k f}) < \phi(f)=\esssup_X f.
\end{equation}
Then, the monetary risk measure $\psi\colon L^\infty\to\mathbb{R}$ defined by
\[
\psi(h) := \limsup_{k\to\infty} \tfrac{1}{n_k} \log \mathcal{E}_{\mu,g_{n_k}}(e^{n_k h})
\]
is maxitive and law invariant. Applying Theorem~\ref{thm:main3} again, we obtain $\psi(f) = \esssup_X f$, which contradicts \eqref{eq:subsequence}.

\end{enumerate}
\end{examples}

\begin{remark}

Theorem~\ref{thm:main2} shows that the maximum loss \({\rm ML}\) is the unique example of a maxitive monetary risk measure that is law-invariant (provided that $\mu$ is atomless and $L^2$ is separable). In addition, as discussed in \cite{jouini}, the assumption that the probability space is atomless is common in financial modeling, where asset prices are typically modeled as continuous variables. In fact, the probability space \((X, \mathcal{X}, \mu)\) is atomless if and only if there exists a random variable \(\xi\) on \((X, \mathcal{X}, \mu)\) with a continuous distribution; see~\cite{delbaen}. Moreover, the condition that \(L^2\) is separable is not restrictive in practice, as it holds whenever we are dealing with countably many random variables, which is typical in financial applications; see \cite{jouini} for further discussion.

These considerations suggest that, in practical matters, the maximum loss is the only law-invariant monetary risk measure that is also maxitive. This rules out other popular law-invariant measures, such as \emph{Value at Risk} (\({\rm VaR}\)). While this might seem limiting, it is conceptually quite natural, as worst-case  risk evaluations are inherently sensitive to specific scenarios \(x \in X\), making them incompatible with law invariance, which depends solely on the distribution of outcomes rather than their  realizations.

However, one can still construct examples of law-invariant monetary risk measures that are maxitive in a weaker sense. This can be achieved by identifying a random variable \(f\) with its quantile function \(q_\alpha(f)\), and ordering them by \emph{usual stochastic dominance}, defined as \(f \le_{\rm st} g\) whenever \(q_\alpha(f) \le q_\alpha(g)\) for all \(\alpha \in (0,1)\) (in other words, replacing states by quantile levels). Then, a monetary risk measure \(\phi\) is said to be maxitive (with respect to \(\le_{\rm st}\)) if
\[
\phi(f \vee_{\rm st} g) \le \phi(f) \vee \phi(g),
\]
where the  \(f \vee_{\rm st} g\) represents a random variable whose quantile function is given  by \(q_\alpha(f \vee_{\rm st} g) = q_\alpha(f) \vee q_\alpha(g)\) for all $\alpha\in(0,1)$. A typical example of such a maxitive monetary risk measure is \({\rm VaR}\). This form of maxitivity has been studied recently in \cite{wang,kupperzapata}.

On the other hand, maxitive monetary risk measures (in the strong sense) offer reasonable  alternatives to traditional VaR. For example, consider the following maxitive analogue of VaR. Let \(X = A \cup B\) be a partition of the sample space, where \(B\) represents a set of stressed scenarios with \(\nu(B) = 1 - \alpha\), for some confidence level \(\alpha \in (0,1)\). The \emph{maxitive VaR} at level \(\alpha\) is defined as
\[
{\rm maxVaR}_\alpha(f) = \esssup_{A} f.
\]
This risk measure corresponds to the penalized maximum loss \({\rm ML}_I\), where the penalty function \(I\) is equal to \(0\) on \(A\) and \(\infty\) on \(A^c\). 

Like traditional VaR, the maxitive VaR estimates the maximum loss within a specified confidence region. However, whereas standard VaR relies on the underlying probability distribution to determine this region, maxitive VaR allows it to be explicitly constructed from selected scenarios. This  enables the design of confidence regions that can better reflect the structure and interdependence of specific risk factors. Also, since maxVaR is maxitive, and thus convex, and positively homogeneous, it is a coherent risk measure. This contrasts with VaR, which does not satisfy coherence. Such risk measures have been studied in detail in \cite{studer}.

\end{remark}

\section{Applications}\label{sec:applications}

{ In this section, by giving a different interpretation to the main results, we explore other applications to large deviations and the limit theory of risk measures.}

\subsection{Sharp large deviation estimates}
In this subsection, let $(\nu_n)_{n\in\N}$ be  a sequence $(\nu_n)_{n\in\N}$ of probability measures on a measurable space $(X,\mathcal{X})$ such that $\nu_n \ll \mu$ for all $n\in\N$ for a given reference probability measure $\mu$ on $\XX$. Given a measurable function \(I \colon \mathcal{X} \to [0, \infty]\), we say that the sequence \((\nu_n)_{n \in \mathbb{N}}\) satisfies the \emph{sharp large deviation principle} (sharp LDP) with rate \(I\) if 
\begin{equation}\label{eq:sharpLDP}
\lim_{n \to \infty} \tfrac{1}{n} \log \nu_n(A) = -\essinf_A I \quad \text{for all } A \in \mathcal{X}.
\end{equation}
In the special case where \(\nu_n\) are the laws of the sample means of a sequence of i.i.d.\ \(d\)-dimensional random variables, Barbe and Broniatowski~\cite{barbe} proved that if the law \(\nu_1\) satisfies certain regularity conditions, then the sequence \((\nu_n)_{n \in \mathbb{N}}\) satisfies a sharp LDP as stated in~\eqref{eq:sharpLDP}. In the following, we apply the results in the previous section to analyze the non-i.i.d.~case  and establish sufficient conditions for the sharp LDP in terms of the Radon–Nikodym derivatives $\frac{{\rm d}\nu_n}{{\rm d}\mu}$.

\begin{remark}
In the particular situation where \(X\) is a topological space, \(\mathcal{X}\) is the Borel \(\sigma\)-algebra \(\mathcal{B}(X)\), and \(I \colon X \to [0, \infty]\) is a lower semicontinuous function, the sequence \((\nu_n)_{n \in \mathbb{N}}\) is said to satisfy the \emph{large deviation principle} (LDP) with rate \(I\) if 
\begin{equation}\label{eq:LDP2}
-\inf_{x \in \operatorname{int}(A)} I(x) \le \liminf_{n \to \infty} \tfrac{1}{n} \log \nu_n(A) \le \limsup_{n \to \infty} \tfrac{1}{n} \log \nu_n(A) \le -\inf_{x \in \operatorname{cl}(A)} I(x)
\end{equation} 
for all \(A \in \mathcal{B}(X)\), 
where \(\operatorname{int}(A)\) and \(\operatorname{cl}(A)\) denote the interior and closure of \(A\), respectively. 
As explained in the introduction, the bounds in \eqref{eq:LDP2} may not be sharp and can sometimes lead to trivial estimates. In contrast, the sharp version of the LDP~\eqref{eq:sharpLDP} provides the exact asymptotics of \(\frac{1}{n} \log \nu_n(A)\). 

Additionally, we have some relations between the sharp LDP \eqref{eq:sharpLDP} and the rough LDP \eqref{eq:LDP2}. 
Suppose that the sequence \((\nu_n)_{n \in \mathbb{N}}\) satisfies the sharp LDP with rate \(I\). Since \(\inf_{x \in F} I(x) \le \essinf_F I\) for all closed sets \(F \subset X\), the upper bound in the rough LDP \eqref{eq:LDP2} holds for all \(A \in \mathcal{B}(X)\). On the other hand, the lower bound in the rough LDP \eqref{eq:LDP2} may not hold for some open sets \(G\) if \(\inf_{x \in G} I(x) < \essinf_G I\). 
However, if \(I \colon X \to [0, \infty]\) is continuous and every \(x \in X\) has a neighborhood base \(\mathcal{U}_x\) such that \(\mu(U) > 0\) for all \(U \in \mathcal{U}_x\), then \(\inf_{x \in G} I(x) = \essinf_G I\) for all open sets \(G \subset X\). Consequently, the lower bound in the rough LDP \eqref{eq:LDP2} holds.
\end{remark}

In the following, we consider the (upper) Varadhan functional \(\phi \colon L^0 \to [-\infty, \infty]\) defined by
\[
\phi(f) := \limsup_{n \to \infty} \tfrac{1}{n} \log \int_X e^{n f} \, \mathrm{d}\nu_n.
\]
The function $\phi$ shares the properties of a monetary risk measure, and  is maxitive as a consequence of the principle of the largest term (see (1) in Examples \ref{examples}). 
We also consider the associated concentration \(J\) defined as in \eqref{eq:concentration}, the minimal penalty \(\I\) defined as in \eqref{eq:minrate0}, and the set \(L^{\phi}\) defined as in~\eqref{eq:unbounded}. For every \(A \in \mathcal{X}\),
\[
J_A := \limsup_{n \to \infty} \tfrac{1}{n} \log \nu_n(A).
\]

As a consequence of Theorem \ref{thm:main1}, we have the following.

\begin{proposition}\label{prop:sharpLDP}
The sequence $(\nu_n)_{n\in\N}$ verifies the sharp LDP \eqref{eq:sharpLDP} with rate $I$ if, and only if, it satisfies the Laplace principle (LP)
\begin{equation}\label{eq:sharpLP}
  \lim_{n\to\infty}\tfrac{1}{n}\log \int_X e^{n f} {\rm d}\nu_n=\esssup_X\{f-I\}\quad\mbox{ for all }f\in L^{\phi}.  
\end{equation}
In that case, $I=\I$~a.s.    
\end{proposition}
\begin{proof}
Suppose that \((\nu_n)_{n\in\N}\) satisfies the sharp LDP with rate  \(I\). In that case, \({J}_A = -\essinf_A I\) for all \(A \in \XX\), and it follows from $(4)\Rightarrow(6)$ in Theorem \ref{thm:main1} that 
\[
{\phi}(f) = \limsup_{n\to\infty} \tfrac{1}{n} \log \int_X e^{n f} \, {\rm d}\nu_n = \esssup_X\{f - I\} \quad \text{for all } f \in L^{\phi},
\]
and  \(I = \I\)~a.s. To establish \eqref{eq:sharpLP}, it remains to show that the limit superior in \(\phi(f)\) is in fact a limit. Fixed \(f \in L^{\phi}\), take a subsequence \(n_1 < n_2 < \cdots\) such that 
\[
\lim_{k\to\infty} \tfrac{1}{n_k} \log \int_X e^{n_k f} \, {\rm d}\nu_{n_k} = \liminf_{n\to\infty} \tfrac{1}{n} \log \int_X e^{n f} \, {\rm d}\nu_n.
\]
The monetary risk measure \(\psi(h) := \limsup_{k\to\infty} \tfrac{1}{n_k} \log \int_X e^{n_k h} \, {\rm d}\nu_{n_k}\) is maxitive by the principle of the largest term, and the associated concentration  \(J^\psi\) satisfies
\[
J^\psi_A = \lim_{k\to\infty} \tfrac{1}{n_k} \log \nu_{n_k}(A) = -\essinf_A I \quad \text{for all } A \in \XX.
\]
Again, by Theorem \ref{thm:main1}, we have 
\[
\psi(h) = \limsup_{n\to\infty} \tfrac{1}{n} \log \int_X e^{n h} \, {\rm d}\nu_n = \esssup_X\{h - I\} \quad \text{for all } h \in L^{\psi}.
\]
In particular, since \(f \in L^{\phi} \subset L^{\psi}\), we obtain
\begin{align*}
\esssup_X\{f - I\} &=\phi(f)=\limsup_{n\to\infty} \tfrac{1}{n} \log \int_X e^{n f} \, {\rm d}\nu_n\\
 &\geq 
\liminf_{n\to\infty} \tfrac{1}{n} \log \int_X e^{n f} \, {\rm d}\nu_n = \psi(f) = \esssup_X\{f - I\}.
\end{align*}
This shows that the limit superior in \(\phi(f)\) is indeed a limit, as required. 

The converse direction, where the LP~\eqref{eq:sharpLP} implies the sharp LDP, follows from the implication \((6) \Rightarrow (4)\) in Theorem~\ref{thm:main1}. The existence of the limit in this case can be established by a subsequence argument similar to the one used above.  
\end{proof}

\begin{remark}
If \(X\) is a Hausdorff regular topological space and \(\mathcal{X} = \mathcal{B}(X)\), Varadhan's integral lemma (see, e.g., \cite[Theorem 4.3.1]{dembo}) asserts that if the LDP~\eqref{eq:LDP2} holds with  rate  \(I\), then the following form of the Laplace principle is satisfied
\[
\lim_{n \to \infty} \tfrac{1}{n} \log \int_X e^{n f} \, \mathrm{d}\nu_n = \sup_{x \in X} \{ f(x) - I(x) \} \quad \text{for all } f \in C_b(X),
\]
where \(C_b(X)\) denotes the set of all bounded continuous functions from \(X\) to \(\mathbb{R}\). 

The converse statement, due to Bryc~\cite{bryc}, is true if \(I\) has compact level sets or if \(X\) is a normal topological space, see also \cite[Theorem 4.4.13]{dembo}. Thus, Proposition~\ref{prop:sharpLDP} can be understood as a non-topological version of the Varadhan-Bryc equivalence between the LDP and the Laplace principle, where the supremum is replaced by an essential supremum.
\end{remark}

The following result is standard. For a proof, refer to~\cite[Lemma 13.1]{aliprantis}.
\begin{lemma}\label{lem:Lpnorm}
For all \(f \in L^\infty_+\), it holds that
\[
\lim_{n \to \infty} \left( \int_X f^n \, \mathrm{d}\mu \right)^{1/n} = \esssup_X f.
\]
\end{lemma}

We establish the following criterion for the existence of a sharp LDP, which seems to be
new.
\begin{theorem}\label{thm:sharpLDP}
Let \(I \colon X \to [0, \infty]\) be a measurable function, and suppose that \((A_k)_{k \in \mathbb{N}} \subset \mathcal{X}\) satisfies the following conditions:
\begin{enumerate}
    \item \(\{I < \infty\} = \bigcup_{k \in \mathbb{N}} A_k\) a.s.
    \item \(\limsup_{n \to \infty} \tfrac{1}{n} \log \nu_n(A_k^c) \to -\infty\) as \(k \to \infty\).
    \item For each \(k \in \mathbb{N}\),
    \[
    \frac{1}{n} \log \tfrac{{\rm d}\nu_n}{{\rm d}\mu} \to -I \quad \text{as } n \to \infty \text{ a.s.~uniformly on } A_k.
    \]
\end{enumerate}
Then, \((\nu_n)_{n \in \mathbb{N}}\) satisfies both the sharp LDP~\eqref{eq:sharpLDP} and the LP~\eqref{eq:sharpLP} with rate \(I\).
\end{theorem}

\begin{proof}
For each \(n \in \mathbb{N}\), define \(h_n := \frac{{\rm d}\nu_n}{{\rm d}\mu}\). Fix \(k \in \mathbb{N}\) and \(\delta \in (0,1)\). By the condition (3), for sufficiently large \(n\), it holds that
\[
e^{-\delta n} \leq h_n e^{n I} \leq e^{\delta n} \quad \text{a.s. on } A_k.
\]
For any \(B \subset A_k\), we obtain
\begin{align*}
\limsup_{n \to \infty} \tfrac{1}{n} \log \nu_n(B) &= \limsup_{n \to \infty} \tfrac{1}{n} \log \int_B h_n \, {\rm d}\mu \\
&= \limsup_{n \to \infty} \tfrac{1}{n} \log \int_B h_n e^{n I} e^{-n I} \, {\rm d}\mu \\
&\leq \limsup_{n \to \infty} \tfrac{1}{n} \log \int_B e^{\delta n} e^{-n I} \, {\rm d}\mu \\
&= \delta + \limsup_{n \to \infty} \tfrac{1}{n} \log \int_B e^{-n I} \, {\rm d}\mu.
\end{align*}
Similarly, we have
\[
\liminf_{n \to \infty} \tfrac{1}{n} \log \nu_n(B) \geq -\delta + \liminf_{n \to \infty} \tfrac{1}{n} \log \int_B e^{-n I} \, {\rm d}\mu.
\]
Letting \(\delta \downarrow 0\), we conclude that the limit superior in \(J_B\) is indeed a limit, and
\begin{equation}\label{eq:SubsetB}
J_B = \lim_{n \to \infty} \tfrac{1}{n} \log \nu_n(B) = \liminf_{n \to \infty} \tfrac{1}{n} \log \int_B e^{-n I} \, {\rm d}\mu.
\end{equation}
Applying Lemma~\ref{lem:Lpnorm}, we obtain
\[
J_B = \lim_{n \to \infty} \left( \int_B e^{-n I} \, {\rm d}\mu \right)^{1/n} = \log \left(\esssup_B e^{-I}\right) = -\essinf_B I.
\]
Since \(B \subset A_k\) was arbitrary, we conclude that the condition (3) in Theorem~\ref{thm:main4} is satisfied.
Also, the conditions (1) and (2) are exactly those in Theorem~\ref{thm:main4}. Therefore, we obtain from that result that
\begin{equation}\label{eq:repPrelim}
J_A = \limsup_{n \to \infty} \tfrac{1}{n} \log \nu_n(A) = -\essinf_A I \quad \text{for all } A \in \mathcal{X},
\end{equation}
and, furthermore, \(\I = I\) a.s.

To complete the proof, it suffices to show that the limit superior in \eqref{eq:repPrelim} is indeed a limit; in that case, both the sharp LDP and the LP hold by Proposition~\ref{prop:sharpLDP}. Suppose by contradiction that \(\liminf_{n \to \infty} \frac{1}{n} \log \nu_n(A) < \beta < -\essinf_A I\) for some \(A \in \mathcal{X}\) and \(\beta \in \mathbb{R}\). In that case, \(\mu(A \cap \{I < -\beta\}) > 0\). By the condition (1), there exists \(k \in \mathbb{N}\) such that \(B := A_k \cap A \cap \{I < -\beta\}\) satisfies \(\mu(B) > 0\). 

Since \(B \subset A_k\), the limit superior in \(J_B\) is a limit by \eqref{eq:SubsetB}, and 
\[
\beta > \liminf_{n \to \infty} \tfrac{1}{n} \log \nu_n(A) \geq \lim_{n \to \infty} \frac{1}{n} \log \nu_n(B) = J_B = -\essinf_B I \geq \beta,
\]
which is a contradiction. The proof is complete.
\end{proof}

The previous result gives us a technique to establish a sharp large deviations principle if it exists. 
Namely, by computing the  $\lim_{n\to \infty}\tfrac{1}{n}\log \frac{{\rm d}\nu_n}{{\rm d}\mu}=-I$, we obtain a candidate for the rate $I$. 
The next step is to exhaust $\{I<\infty\}$ by sets $A_k$ where the convergence is uniform and concentrating low probability mass in the complements $A_k^c$.

\begin{example}
For \(d \in \mathbb{N}\), consider \(X = [1,\infty)^d\) endowed with the Borel \(\sigma\)-field \(\mathcal{B}([1,\infty)^d)\). 
For each \(n \in \mathbb{N}\), let \(\nu_n\) be the law of a multivariate Pareto distribution of the first kind on \([1,\infty)^d\) with shape parameter \(n\), i.e., \(\nu_n\) has the density function 
\[
f_n(x_1,\ldots,x_d)=\frac{n(n+1)\cdots (n+d-1)}{\left(\sum_{i=1}^d x_i - d + 1\right)^{n+d}} \quad \text{for all } x_i \ge 1. 
\]   
In this case, we have 
\[
 \lim_{n\to\infty} \tfrac{1}{n} \log f_n(x_1,\ldots,x_d) = -\log\left(\sum_{i=1}^d x_i - d + 1\right).
\]
This suggests the candidate rate function \(I(x) = \log\left(\sum_{i=1}^d x_i - d + 1\right)\), assuming a sharp LDP holds. 
The convergence is uniform on the sets \(A_k := [1,k]^d\).

Using the fact that the cumulative distribution function of \(\nu_n\) is 
\[ 
F_n(x_1,\ldots,x_d) = 1 - \left(\sum_{i=1}^d x_i - d + 1\right)^{-n}, 
\]
for each \(k \in \mathbb{N}\) we have  
\begin{align*}
J_{A_k^c} &= \limsup_{n\to\infty} \tfrac{1}{n} \log \nu_n(A_k^c) \\
&= \limsup_{n\to\infty} \tfrac{1}{n} \log \left(1 - F_n(k,\ldots,k)\right) \\
&= -\log\left(d(k-1) + 1\right) \to -\infty \quad \text{as } k \to \infty.
\end{align*} 

We have verified the conditions (1)--(3) in Theorem~\ref{thm:sharpLDP}. 
Thus, we obtain the sharp LDP, i.e.,
\[
\lim_{n\to\infty} \tfrac{1}{n} \log \nu_n(A) = -\essinf_{(x_1,\ldots,x_d) \in A} \log\left(\sum_{i=1}^d x_i - d + 1\right) \quad \text{for all } A \in \mathcal{B}([1,\infty)^d).
\]
\end{example}

Theorem~\ref{thm:sharpLDP} provides sufficient conditions for the sharp LDP for general sequences of probability measures. In the following, we revisit the i.i.d.~case under this new criterion. For the remainder of this subsection, we assume that $X = \mathbb{R}^d$, $\mathcal{X} = \mathcal{B}(\mathbb{R}^d)$ (the Borel $\sigma$-algebra), and that $\mu$ has the same null sets as the Lebesgue measure $\lambda$ on $\mathbb{R}^d$. In addition, let $(\xi_n)_{n \in \mathbb{N}}$ be a sequence of i.i.d.~$\mathbb{R}^d$-valued random variables defined on a probability space $(\Omega, \Sigma, P)$, and assume that the sample mean $\bar{\xi}_n = \frac{1}{n}(\xi_1 + \cdots + \xi_n)$ is distributed according to the measure $\nu_n$ for each $n \in \mathbb{N}$.
 The \emph{logarithmic moment generating function} of $\nu_1$ is given by
\[
\Lambda \colon \mathbb{R}^d \to [0, \infty], \quad \Lambda(x) := \log \int_{\mathbb{R}^d} e^{x \cdot y} \nu_1({\rm d}y),
\]
where $x \cdot y$ denotes the standard inner product. Additionally, we define the convex conjugate  $\Lambda^\ast \colon \mathbb{R}^d \to [0, \infty]$ of $\Lambda$ by
\[
\Lambda^\ast(x) := \sup_{y \in \mathbb{R}^d} \{ x \cdot y - \Lambda(y) \}.
\]

\begin{remark}\label{rem:cramer}
If \(\Lambda\) is finite, then Cram\'{e}r's theorem asserts that \((\nu_n)_{n \in \mathbb{N}}\) satisfies the rough LDP~\eqref{eq:LDP2} with the rate  \(\Lambda^\ast\), see \cite[Theorem 2.2.30]{dembo}.  
In this case, the rate function \(\Lambda^\ast\) is `good', meaning that the sublevel set \(\{\Lambda^\ast \leq r\} \subset \mathbb{R}^d\) is compact for all \(r \in \mathbb{R}\). 
Although, in this situation, the rough LDP~\eqref{eq:LDP2} is satisfied,  the sharp LDP may not hold. This means that stronger conditions are needed than those required by Cram\'{e}r's theorem. 
 To give an example, consider the distribution ${\rm Bernoulli}(1/2)$. 
 In that case, the set $A:=\{1/2\}$ verifies $\nu_n(A)=0$ for odd $n$, and $\nu_n(A)=\binom{n}{n/2}\tfrac{1}{2^n}$ for even $n$, and 
 $$-\infty=\liminf_{n\to\infty}\tfrac{1}{n}\log\nu_n(A)<\limsup_{n\to\infty}\tfrac{1}{n}\log\nu_n(A)=0.$$
 Therefore, the sequence $(\tfrac{1}{n}\log\nu_n(A))_{n\in\N}$ does not converge  and, consequently, the sharp LDP cannot be verified. 
\end{remark}

Next, we show that Theorem~\ref{thm:sharpLDP} also covers the i.i.d.~case studied by Barbe and Browniatoski~\cite{barbe}. Specifically, Andriani and Baldi~\cite{baldi} gave the asymptotics for the density of the sample means of a sequence of i.i.d.~random variables in $\mathbb{R}^d$, uniformly over compact sets. By combining the result of Andriani and Baldi with Theorem~\ref{thm:sharpLDP}, we obtain the following result, which aligns with~\cite[Corollary 2.1]{barbe}.
\begin{corollary}\label{cor:iid}
 Suppose  that:
 \begin{enumerate}
     \item $\Lambda$ is finite. 
     \item $\nu_n$ has a bounded density $f_n$ for  $n\in\N$ large enough.
 \end{enumerate}
 Then, $(\nu_n)_{n\in\N}$ verifies the sharp LDP and the LP with rate $\Lambda^\ast$.
\end{corollary}
\begin{proof}
Set \(I := \Lambda^\ast\).  
If the conditions (1) and (2) above are satisfied, \cite[Theorem 3.1]{baldi} asserts that the asymptotic equivalence
\begin{equation}\label{q:Baldi}
f_n(x) \sim \left( \frac{n}{2\pi} \right)^{d/2} \left| I''(x) \right|^{1/2} e^{-n I(x)} \quad \text{as } n \to \infty,
\end{equation}
holds uniformly on compact subsets of the effective domain \(\{ I < \infty \}\), where \(I''\) denotes the Hessian of \(I\), and \(\left| I''(x) \right|\) denotes the determinant of \(I''(x)\).\footnote{\cite[Theorem 3.1]{baldi} considers instead the \emph{admissible domain}, which is a different set from the \emph{effective domain} (the admissible domain is a subset of the effective domain). However, under our assumption that \( \Lambda \) is finite, both sets coincide.
}  

As pointed out in \cite{baldi}, the function \(I\) is \(C^\infty\), and \(I''\) is strictly positive definite on the effective domain. By continuity, for every compact set \(K \subset \{ I < \infty \}\), we can find constants \(0 < a < b\) such that \(a < \left| I''(x) \right| < b\) for all \(x \in K\). 
 Combining this with the asymptotic equivalence in \eqref{q:Baldi}, we obtain that
\begin{equation}\label{q:Baldi2}
\tfrac{1}{n} \log(f_n(x)) \to -I(x) \quad \text{as } n \to \infty,
\end{equation}
uniformly on compact subsets of \(\{ I < \infty \}\).

Define \(A_k := \{ I \leq k \}\) for every \(k \in \mathbb{N}\). By Remark~\ref{rem:cramer}, the sets \((A_k)_{k \in \mathbb{N}}\) are compact subsets of \(\{ I < \infty \}\). As noted in Remark~\ref{rem:cramer}, the sequence \((\nu_n)_{n \in \mathbb{N}}\) satisfies the LDP with rate  \(I\). Hence, in particular, 
\[
J_{A_k^c} \leq J_{\{ I \geq k \}} \leq -\inf_{x \in \{ I \geq k \}} I(x) \leq -k.
\]
Letting \(k \to \infty\) yields \(\lim_{k \to \infty} J_{A_k^c} = -\infty\). This shows that \(I\) and the sequence \((A_k)_{k \in \mathbb{N}}\) satisfy the conditions (1)--(3) of Theorem~\ref{thm:sharpLDP}. Consequently, \((\nu_n)_{n \in \mathbb{N}}\) satisfies the sharp LDP with rate \(I\), as required. 
\end{proof}

\begin{remark}
As mentioned, Corollary \ref{cor:iid} is similar to \cite[Corollary 2.1]{barbe}. The result in \cite{barbe} is proven under Assumptions (1.1), (2.1), and (3.1) in \cite{barbe}. These conditions are rather  technical; for instance, Assumption (2.1) requires the existence of a homeomorphism between the topological interior of the effective domain of \( \Lambda^\ast \) and the support of \( \nu_1 \). 

In contrast, Corollary \ref{cor:iid} assumes simpler conditions: that the averages have bounded densities and that $\Lambda$ is finite. We point out that Assumption (2.2) in \cite{barbe} already implies the existence of densities for the averages, as explained in Comment 2.6 in \cite{barbe}. 

Finally, Corollary \ref{cor:iid} provides also the LP which is not obtained in \cite{barbe}.
\end{remark}

\subsection{A limit result for distorted expectations} 
In this subsection, let \((\nu_n)_{n \in \mathbb{N}}\) be a sequence of probability measures on a measurable space \((X, \mathcal{X})\) such that \(\nu_n \ll \mu\) for all \(n \in \mathbb{N}\), where \(\mu\) is a given reference probability measure on \((X, \mathcal{X})\). Given a concave distortion  \(g: [0, 1] \to [0, 1]\), we define the monetary risk measure \(\phi: L^0 \to [-\infty, \infty]\) by
\begin{equation}\label{eq:distrotedVaradhan}
\phi(f) := \limsup_{n \to \infty} \tfrac{1}{n} \log \mathcal{E}_{\nu_n, g} (e^{nf}),
\end{equation}
where \(\mathcal{E}_{\nu_n, g}(f) := \int_0^\infty g \circ \nu_n(f > x) \, \mathrm{d}x\) is the distorted expectation associated with \(g\) (see (3) in Example \ref{examples}). Since \(g\) is concave, \(\mathcal{E}_{\nu_n, g}\) is subadditive for all \(n\), and \(\phi\) is maxitive by the principle of the largest term (see (1) in Example \ref{examples}). 
The  concentration function \(J\) associated with $\phi$ is given by
\[
J_A := \limsup_{n \to \infty} \tfrac{1}{n} \log(g \circ \nu_n(A)) \quad \text{for all } A \in \mathcal{X}.
\]

\begin{lemma}\label{lem:pconcave}
  If $h\colon [0,1]\to[0,1]$ is a concave distortion, then the right derivative $h^\prime_+(0)$ exists as an extended real number in $[1,\infty]$. 
Moreover, if $h^\prime_+(0)<\infty$, then $h(x)\le x\cdot h^\prime_+(0)$ for all $x\in(0,1]$.
\end{lemma}
\begin{proof}
All the assertions follow from the fact that the function $x\mapsto \tfrac{h(x)}{x}$ is non-decreasing on $(0,1]$.  To prove that, assume that  $0<x_1<x_2\le 1$. Then,  by concavity, 
$$h(x_1)=h\left(\tfrac{x_1}{x_2}x_2+(1-\tfrac{x_1}{x_2})0\right)\ge \tfrac{x_1}{x_2}h(x_2).$$ Dividing by $x_1$, we get the $\tfrac{h(x_1)}{x_1}\ge \tfrac{h(x_2)}{x_2}$, as desired.
\end{proof}

\begin{definition}
The distortion $g$ is said to have \emph{vanishing order} $p$, with $p\ge 1$, if the distortion $h:=g^p$ is concave and verifies $h^\prime_+(0)<\infty$. 
\end{definition}

The vanishing order  of the distortion $g$ measures how fast the function $g(x)$ falls to $0$ as as $x\downarrow 0$. More specifically, if $g$ has vanishing order $p$, then
\[
\lim_{x\downarrow 0}\frac{g(x)}{x^{1/p}}=\lim_{x\downarrow 0}\frac{h(x)^{1/p}}{x^{1/p}}
=\lim_{x\downarrow 0}\left(\frac{h(x)}{x}\right)^{1/p}=(h^\prime_+(0))^{1/p}\in [1,\infty).
\]
Since the limit $(h^\prime_+(0))^{1/p}$ is finite and not null, then $g$  behaves asymptotically like $x^{1/p}$ near $0$.

\begin{theorem}\label{thm:distorted}
Suppose that $g$ has vanishing order $p$ .  
 If the sequence $(\nu_n)_{n\in\N}$ satisfies the sharp LDP \eqref{eq:sharpLDP} with rate $I$, then 
\[
\lim_{n\to\infty} \tfrac{1}{n}\log \EE_{\nu_n,g}(e^{n f})=\esssup_X(f-\tfrac{1}{p}I)\quad\mbox{ for all }f\in L^{\phi}.
\]
\end{theorem}
\begin{proof}
Set  \( h := g^p \). Given \( A \in \mathcal{X} \), by Lemma~\ref{lem:pconcave}, we have
\begin{align*}
\limsup_{n \to \infty} \tfrac{1}{n} \log (g \circ \nu_n(A))
&= \limsup_{n \to \infty} \tfrac{1}{np} \log (h \circ \nu_n(A)) \\
&\leq \limsup_{n \to \infty} \tfrac{1}{np} \log \left(\nu_n(A) \cdot h'_+(0)\right) \\
&\leq \limsup_{n \to \infty} \tfrac{1}{np} \log (\nu_n(A)) 
    + \limsup_{n \to \infty} \tfrac{1}{np} \log (h'_+(0)) \\
&= -\essinf_A \tfrac{1}{p} I,
\end{align*}
where we have used that \((\nu_n)_{n \in \mathbb{N}}\) satisfies the sharp LDP with rate  \(I\). On the other hand, by concavity, \( h(x) \geq x \) for all \( x \in [0,1] \). Hence,
\begin{align*}
\liminf_{n \to \infty} \tfrac{1}{n} \log (g \circ \nu_n(A))
&= \liminf_{n \to \infty} \tfrac{1}{np} \log (h \circ \nu_n(A)) \\
&\geq \liminf_{n \to \infty} \tfrac{1}{np} \log (\nu_n(A)) \\
&= \tfrac{1}{p} \lim_{n \to \infty} \tfrac{1}{n} \log (\nu_n(A)) \\ 
&= -\essinf_A \tfrac{1}{p} I.
\end{align*}
It follows that
\[
J_A = \lim_{n \to \infty} \tfrac{1}{n} \log (g \circ \nu_n(A)) = -\essinf_A \tfrac{1}{p} I \quad \text{for all } A \in \mathcal{X}.
\]
Then, by Theorem~\ref{thm:main1}, we obtain
\[
\phi(f) = \limsup_{n \to \infty} \tfrac{1}{n} \log \mathcal{E}_{\nu_n, g}(e^{n f}) = \esssup_X \left(f - \tfrac{1}{p} I \right) \quad \text{for all } f \in L^{\phi}.
\]

Finally, a subsequence argument similar to that in the proof of Proposition~\ref{prop:sharpLDP} shows that the limit superior above is indeed a limit.
\end{proof}

Tsanakas~\cite{tsanakas} introduced the \emph{distortion-exponential} insurance premium principle. For an insurance claim whose random outcome is given by a real-valued random variable \(\xi\), defined on a probability space \((\Omega, \Sigma, P)\), the premium to be paid for the claim is 
\begin{equation}\label{eq:dist-Exp}
\Pi_{g, \gamma}(\xi) := \tfrac{1}{\gamma} \log \mathcal{E}_{P, g}\left(e^{\gamma \xi}\right),
\end{equation}
where \(\gamma > 0\) is the risk aversion parameter, and \(g\) is a concave distortion  reflecting the insurer's attitude towards risk. 
In the special case where \(g(x) = x\) is the identity function, we recover the usual exponential premium principle
\begin{equation}\label{eq:Exp}
\Pi_{\gamma}(\xi) := \tfrac{1}{\gamma} \log \mathbb{E}_{P}\left(e^{\gamma \xi}\right).
\end{equation}
We now use Theorem~\ref{thm:distorted} to study the asymptotic behavior of the distortion-exponential principle \eqref{eq:dist-Exp} under the pooling of risks. Consider a homogeneous portfolio of \(n\) insurance claims, whose uncertain outcomes are represented by the real-valued random variables \(\xi_1, \xi_2, \ldots, \xi_n\). According to \eqref{eq:dist-Exp}, the premium to be paid for the entire portfolio is \(\Pi_{g, \gamma}(\xi_1 + \cdots + \xi_n)\). Hence, the premium for each individual contract in this portfolio is given by
\[
\pi_n = \tfrac{1}{n} \Pi_{g, \gamma}(\xi_1 + \cdots + \xi_n).
\]
The following result gives the asymptotics of the individual premium as the number of claims tend to infinity.

\begin{corollary}\label{cor:premium}
Suppose that \( g \) is a concave distortion  with vanishing order \( p \). Consider a sequence of insurance contracts \( \xi_1, \xi_2, \dots \) such that the laws \( \nu_n \) of the averages \( \bar{\xi}_n \) 
satisfy the sharp LDP with rate \( I \). Assume that there exists \( t > 1 \) such that
\begin{equation}\label{eq:premiumCond}
\limsup_{n \to \infty} \tfrac{1}{n} \Pi_{g, \gamma}\left(t(\xi_1 + \xi_2 + \cdots + \xi_n)\right) < \infty.
\end{equation}
Then, 
\begin{equation}\label{eq:premium}
\lim_{n \to \infty} \pi_n = \esssup_{x \in \mathbb{R}} \left( x - \tfrac{1}{p\gamma} I(x) \right).
\end{equation}
Moreover, if the sequence \( \xi_1, \xi_2, \ldots \) is i.i.d.~and satisfies the conditions (1) and (2) in Corollary~\ref{cor:iid}, then
\begin{equation}\label{eq:premiumiid}
\lim_{n \to \infty} \pi_n = \Pi_{p\gamma}(\xi).
\end{equation}
\end{corollary}

\begin{proof} 
To apply Theorem \ref{thm:distorted}, suppose \( X = \mathbb{R} \), \( \mathcal{X} = \mathcal{B}(\mathbb{R}) \), and that \( \mu\) has the same null sets as the Lebesgue measure on $\mathbb{R}$. Consider the monetary risk measure $\phi(f) := \limsup_{n \to \infty} \tfrac{1}{n} \log \mathcal{E}_{\nu_n, g}\left(e^{n f}\right)$, defined as in \eqref{eq:distrotedVaradhan}.  
By \eqref{eq:premiumCond}, the  function \( f_\gamma\colon \mathbb{R} \to \mathbb{R} \),  \( f_\gamma(x) = \gamma x \), satisfies \( f_\gamma \in L^\phi \). Then, it follows from Theorem \ref{thm:distorted} that
\begin{align*}
\lim_{n \to \infty} \pi_n &= \lim_{n \to \infty} \tfrac{1}{n} \Pi_{g, \gamma}(\xi_1 + \xi_2 + \cdots + \xi_n) \\
&= \tfrac{1}{\gamma}\lim_{n \to \infty} \tfrac{1}{n} \log \int_0^\infty g \circ \mathbb{P}\left(e^{n f_\gamma(\bar{\xi}_n)} > x\right) \, \mathrm{d}x \\
&= \tfrac{1}{\gamma}\phi(f_\gamma) = \esssup_{x \in \mathbb{R}} \left(x - \tfrac{1}{p\gamma} I(x)\right),
\end{align*} 
obtaining \eqref{eq:premium}. 
Now, if the sequence \( \xi_1, \xi_2, \ldots \) is i.i.d. and satisfies the conditions (1) and (2) of Corollary~\ref{cor:iid}, then the sequence $(\nu_n)_{n\in\N}$ verifies the sharp LDP with rate $\Lambda^\ast$. Hence, in that case
\[
\lim_{n \to \infty} \pi_n =\esssup_{x \in \mathbb{R}} \left(x - \tfrac{1}{p\gamma } \Lambda^\ast(x)\right).
\]
As pointing out in the proof of Corollary~\ref{cor:iid},  \( \Lambda^\ast \) is continuous on its effective domain. Thus, the essential supremum above is actually a supremum and,
\begin{align*}
\lim_{n \to \infty} \pi_n &= \sup_{x \in \mathbb{R}} \left(x - \tfrac{1}{p\gamma } \Lambda^\ast(x)\right) = \tfrac{1}{p\gamma} \sup_{x \in \mathbb{R}} \left(p\gamma  x - \Lambda^\ast(p\gamma  x)\right) = \tfrac{1}{p\gamma} \Lambda^{\ast\ast}(p\gamma) = \tfrac{1}{\gamma p} \Lambda(p\gamma),
\end{align*}
where in the last equality we have applied the Fenchel-Moreau theorem (see, e.g., \cite[Proposition A.6]{follmer}). The conclusion follows by noting that \( \Pi_{p\gamma}(\xi) = \frac{1}{p\gamma} \Lambda(p\gamma) \).
\end{proof}
Equality \eqref{eq:premiumiid} reveals that as the number of independent claims in a homogeneous portfolio grows, the individual distortion-exponential premium converges to the usual exponential premium  for a single claim, where the risk aversion parameter $\gamma p$ is determined by the vanishing order $p$ of the distortion  $g$.

\begin{appendix}

\section{Auxiliary results}\label{sec:A1}
In this section, we denote by \(\bar{L}^0\) the set of all (equivalence classes modulo \(\mu\)-a.s. equality) of $[-\infty,\infty]$-valued measurable functions on \(X\).  Given \(\mathcal{F} \subset \bar{L}^0\), we define
\[
U(\mathcal{F}) = \{ g \in \bar{L}^0 \colon f \leq g \text{ for all } f \in \mathcal{F} \}.
\]
\begin{theorem}[Föllmer, Theorem A.32]\label{thm:esssup}
Let \(\mathcal{F}\) be a non-empty subset of \(\bar{L}^0\). Then, there exists a unique element \(f^* \in U(\mathcal{F})\) such that \(f^* \leq g\) for all \(g \in U(\mathcal{F})\). If, in addition, \(\mathcal{F}\) satisfies \(f \vee g \in \mathcal{F}\) whenever \(f, g \in \mathcal{F}\), then there exists a sequence \(f_1 \leq f_2 \leq \ldots\) in \(\mathcal{F}\) such that \(f_n \uparrow f^*\) a.s.
\end{theorem}

Given a non-empty set \(\mathcal{F} \subset \bar{L}^0\), the measurable function \(f^*\) characterized in Theorem \ref{thm:esssup} is called the \emph{essential supremum} of \(\mathcal{F}\) and is denoted by \(\esssup \mathcal{F}\), and it is unique up to a.s.~equality. For the empty set, we define \(\esssup \emptyset = -\infty\) by convention. The \emph{essential infimum} of \(\mathcal{F}\) is defined by \(\essinf \mathcal{F} = -\esssup \{-f \colon f \in \mathcal{F}\}\).

\begin{proposition}\label{prop:intSup}
Suppose that \(f \colon X \to \mathbb{R} \cup \{-\infty\}\) is a measurable function that is bounded from above. Then
\[
\sup_{\nu \in \mathcal{M}_1(\mu)} \int_X f \, {\rm d}\nu = \esssup_X f.
\]
\end{proposition}
\begin{proof}
Clearly, \(\sup_{\nu \in \mathcal{M}_1(\mu)} \int_X f \, {\rm d}\nu \leq \esssup_X f\). Suppose, for the sake of contradiction, that
\begin{equation}\label{eq:sup}
\sup_{\nu \in \mathcal{M}_1(\mu)} \int_X f \, {\rm d}\nu < \esssup_X f.
\end{equation}
Then there exists \(A \in \mathcal{X}\) with \(\mu(A) > 0\) such that
\begin{equation}\label{eq:supInt}
\sup_{\nu \in \mathcal{M}_1(\mu)} \int_X f \, {\rm d}\nu < f \quad \text{on } A.
\end{equation}
Since \(\mu(\cdot \mid A) \in \mathcal{M}_1(\mu)\), we have
\[
\sup_{\nu \in \mathcal{M}_1(\mu)} \int_X f(x) \, {\rm d}\nu \geq \int_X f(x) \, \mu({\rm d}x \mid A).
\]
On the other hand, \eqref{eq:supInt} implies that
\[
\sup_{\nu \in \mathcal{M}_1(\mu)} \int_X f(x) \, {\rm d}\nu < \int_X f(x) \, \mu(\textbf{d}x \mid A),
\]
which is a contradiction.
\end{proof}

\section{Proofs}\label{sec:A2}
In this subsection, we prove the results presented in Section \ref{sec:mainResults}. To do so, we adopt some notational conventions. First, we set $(\pm\infty)\cdot 0 := 0$, so that for any two functions \(f, g \colon X \to [-\infty, \infty]\), the function \(f1_A + g1_{A^c}\) assumes the same values as \(f\) on \(A \subset X\) and the same values as \(g\) on \(A^c\). Second, we define \(\essinf_A f = \infty\) and \(\esssup_A f = -\infty\) whenever \(\mu(A) = 0\). 
In line with Section \ref{sec:mainResults}, we fix a monetary risk measure \(\phi\colon L^0 \to [-\infty, \infty]\), define its associated concentration function \(J\) as in \eqref{eq:concentration}, the minimal penalty as in \eqref{eq:minrate0}, and the set \(L^\phi\) as in \eqref{eq:unbounded}.

The following lemma describes the relationship between \(J\) and \(\I\).

\begin{lemma}\label{lem:IJ}
The following equality holds:
\begin{equation}\label{eq:relationIJ}
-\I = \essinf \{J_A \cdot 1_A \colon A \in \XX\}\mbox{ a.s.}
\end{equation}
Moreover, for any measurable function \(I \colon X \to [0, \infty]\), we have:
\begin{enumerate}
    \item If \(J_A \leq -\essinf_A I\) for all \(A \in \XX\), then \(\I \geq I\)~a.s.
    \item If \(J_A \geq -\essinf_A I\) for all \(A \in \XX\), then \(\I \leq I\)~a.s.
\end{enumerate}
\end{lemma}

\begin{proof}
Fix \(A \in \XX\) and \(M \in \mathbb{N}\). By the definition of \(\I\), we have
\[
\I \geq -M 1_{A^c} - \phi(-M 1_{A^c})\mbox{ a.s.}
\]
Multiplying both sides by \(-1_A\), we obtain
\[
-\I 1_A \leq \phi(-M 1_{A^c}) 1_A\mbox{ a.s.}
\]
Letting \(M \to \infty\), we get
\[
-\I 1_A \leq J_A 1_A\mbox{ a.s.}
\]
Since \(-\I \leq 0\), it follows that
\[
-\I \leq J_A\mbox{ a.s.}
\]
Taking the essential infimum over all \(A \in \XX\), we obtain
\[
-\I \leq \essinf \{J_A \cdot 1_A \colon A \in \XX\}\mbox{ a.s.}
\]

To prove the other inequality, fix \(f \in L^\infty\). Suppose \((s_n)\) is a sequence of simple functions such that \(s_n \uparrow f\) a.s. For a fixed \(n \in \mathbb{N}\), express \(s_n\) as
\[
s_n = \sum_{k=1}^m a_k 1_{A_k},
\]
where \((A_k)_{k=1}^m\) is a finite partition of \(X\). Given \(r \in \mathbb{R}\) with \(r < -\|f\|_\infty\), for each \(k \in \{1, 2, \ldots, m\}\), we have
\[
\phi(f) \geq \phi(a_k 1_{A_k} + r 1_{A_k^c}) = a_k + \phi((r - a_k) 1_{A_k^c}).
\]
Letting \(r \to -\infty\), we obtain
\[
\phi(f) \geq a_k + J_{A_k} \quad \text{for all } k.
\]
Multiplying by \(1_{A_k}\) and summing over all \(k\), we get
\[
\phi(f) \geq s_n + \sum_{k=1}^m J_{A_k} 1_{A_k} \geq s_n + \essinf \{J_A \cdot 1_A \colon A \in \XX\}\mbox{ a.s.}
\]
Letting \(n \to \infty\) yields
\[
-\essinf \{J_A \cdot 1_A \colon A \in \XX\} \geq f - \phi(f)\mbox{ a.s.}
\]
Taking the essential supremum over all \(f \in L^\infty\), we obtain
\[
-\essinf \{J_A \cdot 1_A \colon A \in \XX\} \geq \I\mbox{ a.s.,}
\]
completing the proof of~\eqref{eq:relationIJ}.

To prove (1), suppose that \(J_A \leq -\essinf_A I\) for all \(A \in \XX\). 
Take a sequence \((s_n)\) of non-negative simple functions such that \(s_n \uparrow I\) a.s. 
For a fixed \(n \in \mathbb{N}\), express \(s_n\) as 
\[
s_n = \sum_{k=1}^m a_k 1_{A_k},
\]
where \((A_k)_{k=1}^m\) is a finite partition of \(X\) and \(a_k \geq 0\). 
Then, 
\[
s_n \leq \sum_{k=1}^m \left(\essinf_{A_k} I\right) 1_{A_k} \leq \sum_{k=1}^m (-J_{A_k}) 1_{A_k} \leq \I\mbox{ a.s.,}
\]
where we have applied \eqref{eq:relationIJ} in the last inequality. 
Letting \(n \to \infty\), we obtain \(I \leq \I\)~a.s. as desired. 

To prove (2), assume that \(J_A \geq -\essinf_A I\) for all \(A \in \XX\). Fixed \(A \in \XX\), we have
\[
J_A \cdot 1_A \geq \left(-\essinf_A I\right)1_A \geq -I\mbox{ a.s.}
\]
Taking the essential infimum over all \(A \in \XX\), we reach the conclusion.
\end{proof}

Next, we prove Theorem~\ref{thm:main1}.

\begin{proof}
\textbf{$(1) \Rightarrow (2)$:} Assume \((f_n)_{n \in \mathbb{N}} \subset L^\infty\) and \(\sup_{n \in \mathbb{N}} f_n \in L^\infty\). For each \(n \in \mathbb{N}\), define \(g_n := \bigvee_{i=1}^n f_i\). Since \(\phi|_{L^\infty}\) is maxitive, it follows that 
\[
\phi(g_n) = \bigvee_{i=1}^n \phi(f_i) \quad \text{for all } n \in \mathbb{N}.
\]
Given that \(\phi|_{L^\infty}\) is also continuous from below, and since \(g_n \uparrow \sup_{n \in \mathbb{N}} f_n\) a.s. as \(n \to \infty\), we can take the limit as \(n \to \infty\) to obtain 
\[
\phi\left(\sup_{n \in \mathbb{N}} f_n\right) = \sup_{n \in \mathbb{N}} \phi(g_n) = \sup_{n \in \mathbb{N}} \phi(f_n).
\]

\textbf{$(2) \Rightarrow (1)$:} It is simple to verify.

\textbf{$(1) \Rightarrow (3)$:} Since \(\phi|_{L^\infty}\) is maxitive, it follows that \(\phi|_{L^\infty}\) is convex; see \cite[Proposition 2.1]{kupper}. Define \(\mathcal{A} := \{f \in L^\infty : \phi(f) \leq 0\}\). The maxitivity of \(\phi\) implies that \(f \vee g \in \mathcal{A}\) for all \(f, g \in \mathcal{A}\). Moreover, by the definition of \(\I\) and property (T), we have \(\I = \esssup \mathcal{A}\).
Applying Theorem~\ref{thm:esssup}, there exists a sequence \(f_1 \leq f_2 \leq \dots\) in \(\mathcal{A}\) such that \(f_n \uparrow \I\) a.s. Now, fix \(\nu \in \MM_1\). Since \(\nu \ll \mu\), we have \(f_n \uparrow \I\) \(\nu\)-a.s. Therefore, by the monotone convergence theorem, we obtain
\[
(\phi|_{L^\infty})^\ast(\nu) = \sup_{f \in \mathcal{A}} \int_X f \, {\rm d}\nu \geq \sup_n \int_X f_n \, {\rm d}\nu = \int_X \I \, {\rm d}\nu.
\]
On the other hand, since \(f \leq \I\) \(\mu\)-a.s. and \(\nu\)-a.s. for all \(f \in \mathcal{A}\), we have
\[
(\phi|_{L^\infty})^\ast(\nu) = \sup_{f \in \mathcal{A}} \int_X f \, {\rm d}\nu \leq \int_X \I \, {\rm d}\nu.
\]
This completes the proof of (3).

\textbf{$(3) \Rightarrow (5)$:} Fix \(f \in L^\infty\). Since \(\phi|_{L^\infty}\) is convex and continuous from below,\footnote{Recall that the definition of a monetary risk measure is taken from~\cite{follmer} with a sign change. Thus, continuity from below, as defined here, corresponds to continuity from above in \cite{follmer}.} by \cite[Theorem 4.33]{follmer}, we have
\[
\phi(f) = \sup_{\nu \in \MM_1} \left\{ \int_X f \, {\rm d}\nu - (\phi^\ast|_{L^\infty})(\nu) \right\} = \sup_{\nu \in \MM_1} \left\{ \int_X f \, {\rm d}\nu - \int_X \I \, {\rm d}\nu \right\}
\]
\[
 = \sup_{\nu \in \MM_1} \left\{ \int_X (f - \I) \, {\rm d}\nu \right\}
= \esssup_X(f - \I) = {\rm ML}_{\I}(f),
\]
where we have applied Proposition~\ref{prop:intSup}. This proves (5).

\textbf{$(5) \Rightarrow (1)$:} The penalized maximum loss \({\rm ML}_{\I}\) is clearly maxitive and continuous from below.

\textbf{$(5) \Rightarrow (4)$:}  
If (5) holds, then \(\phi\) is clearly maxitive. Fix \(A \in \XX\). For any \(r < 0\), we have
\[
\phi(r 1_{A^c}) = \esssup_X(r 1_{A^c} - \I) \geq \esssup_X(-\infty \cdot 1_{A^c} - \I) = -\essinf_A \I.
\]
Letting \(r \to -\infty\), it follows that \(J_A \geq -\essinf_A \I\). To prove the other inequality, fix \(M > 0\), and set \(\alpha := (\essinf_A \I) \wedge M \in [0, \infty)\). We have \(\mu\left(A \cap \{\I < \alpha\}\right) = 0\). Consequently, for a given \(r < 0\), it holds that 
\[
r 1_{A^c} - \I \leq r \quad \text{ a.s. on } \{\I < \alpha\}.
\]
Moreover, since \(r 1_{A^c} \leq 0\), we have 
\[
r 1_{A^c} - \I \leq -\alpha \quad \text{ a.s. on } \{\I \geq \alpha\}.
\]
Choosing \(r < -\alpha\), it follows that \(r 1_{A^c} - \I \leq -\alpha\) a.s. Thus,
\[
J_A \leq \phi(r 1_{A^c}) = \esssup_X(r 1_{A^c} - \I) \leq -\alpha = (-\essinf_A \I) \vee -M.
\]
Letting \(M \to \infty\), we obtain \(J_A \leq -\essinf_A \I\), and (4) follows.

\textbf{$(4) \Rightarrow (5)$:} To show \eqref{eq:concentrationRep}, we first assume that \(f \in L^\infty\) is a simple function, that is, \(f = \sum_{i=1}^N a_i 1_{\{f = a_i\}}\) with \(a_1 < a_2 < \dots < a_N\). In that case, 
\[
f = \bigvee_{1 \le i \le N} \left(a_i 1_{\{a_i \ge f > a_{i-1}\}} + r 1_{\{a_i \ge f > a_{i-1}\}^c}\right) 
\]
for \(r < -\|f\|_\infty\), where we define \(a_0 = -\infty\). 
Since \(\phi\) is maxitive, we have
\[
\phi(f) = \bigvee_{1 \le i \le N} \phi\left(a_i 1_{\{a_i \ge f > a_{i-1}\}} + r 1_{\{a_i \ge f > a_{i-1}\}^c}\right) = \bigvee_{1 \le i \le N} \left(a_i + \phi\left((r - a_i) 1_{\{a_i \ge f > a_{i-1}\}^c}\right)\right).
\]
Letting \(r \to -\infty\), we get
\[
\phi(f) = \bigvee_{1 \le i \le N} \left(a_i + J_{\{a_i \ge f > a_{i-1}\}}\right) = \bigvee_{1 \le i \le N} \left(a_i - \essinf_{\{a_i \ge f > a_{i-1}\}} \I\right)
\]
\[
 = \bigvee_{1 \le i \le N} \esssup_{\{a_i \ge f > a_{i-1}\}} \left(a_i - \I\right) = \esssup_X \left(f - \I\right),
\]
where the second equality follows from (4). Since the simple functions are dense in \(L^\infty\) and \(\phi\) is Lipschitz continuous (see \cite[Lemma 4.3]{follmer}) the result follows.

\textbf{$(5) \Rightarrow (6)$:} We assume that  $\phi$ is maxitive on $L^0$. Fix $f\in L^\phi$. 
By definition of $L^\phi$ we can take $t>1$ such that $\phi(tf)<\infty$.   
Fix $m\in\N$.  
Given $h\in L^0_+$ with $-h\le t(f-m)$, we have
\begin{align*}
-m + \phi(f 1_{\{f\ge m\}}-h 1_{\{f<m\}})&= \phi\big((f-m) 1_{\{f\ge m\}}-(h+m) 1_{\{f<m\}}\big)\\
&\le \phi\left(t(f-m) 1_{\{f\ge m\}}-h 1_{\{f<m\}}\right)\\
&\le \phi(t(f-m))\\
&= -mt + \phi(tf).
\end{align*}
Taking the infimum over all $h\in L^0_+$ with $-h\le t(f-m)$, we have 
\[
\inf_{h\in L^0_+} \phi\left(f 1_{\{f\ge m\}}-h 1_{\{f<m\}}\right)\le (1-t)m + \phi(tf).
\]
Since $t>1$ and $\phi(tf)<\infty$, the right hand side converges to $-\infty$ as $m\to\infty$.  
 Letting $m\to\infty$ results in 
\begin{equation}\label{eq:tail}
\lim_{m\to\infty} \inf_{h\in L^0_+} \phi\left(f 1_{\{f\ge m\}}-h 1_{\{f<m\}}\right)
=-\infty.
\end{equation}

In the following, for  $a,b\in[-\infty,\infty]$, $a<b$, we define the truncated function $f_{a,b}:=a\vee (f\wedge b)$.  
Fixed $m,n\in\N$, by the maxitivity of $\phi$, we have that 
\[
\phi(f)\le\phi(f_{-m,\infty})\le \phi(f_{-m,n})\vee 
\inf_{h\in L^0_+} \phi\left(f 1_{\{f\ge n\}}-h 1_{\{f<n\}}\right).
\]
Since $f_{-m,n}\in L^\infty$, applying (5) we have
\begin{align*}
\phi(f)&\le {\rm ML}_I(f_{-m,n})\vee 
\inf_{h\in L^0_+} \phi\left(f 1_{\{f\ge n\}}-h 1_{\{f<n\}}\right)\\
&\le  {\rm ML}_I(f_{-m,\infty})\vee 
\inf_{h\in L^0_+} \phi\left(f 1_{\{f\ge n\}}-h 1_{\{f<n\}}\right).
\end{align*}
Letting $n\to\infty$, it follows from \eqref{eq:tail} that
\begin{align*}
\phi(f) &\le {\rm ML}_I(f_{-m,\infty})= {\rm ML}_I(-m\vee f)= (-m)\vee{\rm ML}_I(f). 
\end{align*} 
Now, letting $m\to\infty$, we get 
\[
\phi(f) \le {\rm ML}_I(f).
\]
On the other hand, fixed $n,m\in\N$, since $\phi$ is maxitive
\begin{align*}
(-m) \vee \phi(f)&=\phi\big((-m) \vee f\big)\ge \phi\left(f_{-m,n}\right)
={\rm ML}_I(f{-m,n})\ge {\rm ML}_I(f{-\infty,n}).
\end{align*}
Taking the limit as $m\to\infty$, we get
\[
\phi(f)\ge {\rm ML}_I(f{-\infty,n}).
\]
Letting $n\to\infty$, we finally obtain $\phi(f)\ge {\rm ML}_I(f)$, which proves (6).

\textbf{$(6) \Rightarrow (5)$:} It is obvious.
 
Finally, if (4) is satisfied by \(I\), then \(I = \I\) a.s. by Lemma~\ref{lem:IJ}.
If (3), (5), or (6) are satisfied by \(I\), then, following the arrows in the proof above,  (4) must also be satisfied by \(I\), implying that \(I = \I\)~a.s. 
This completes the proof of Theorem~\ref{thm:main1}.
\end{proof}

We now turn to the proof of Theorem~\ref{thm:main2}.

\begin{proof}
\textbf{\((1) \Rightarrow (2)\):} By the monotonicity and translation invariance of \(\phi\), we have
\[
\phi(f) \leq \phi\left(\esssup_X f\right) = \phi(0) + \esssup_X f = \esssup_X f = {\rm ML}(f)
\]
for all \(f \in L^\infty\).

On the other hand, if \(J_A = 0\) whenever \(\mu(A) > 0\), then by Lemma~\ref{lem:IJ}, 
\[
\I = -\essinf\{J_A 1_A : A \in \XX\} = 0\mbox{ a.s.}
\]
Therefore,
\[
{\rm ML}(f) = \esssup_X f = \esssup_X(f - \I) \leq \phi(f) \quad \text{for all } f \in L^\infty,
\]
which implies \((2)\).

\textbf{\((2) \Rightarrow (1)\):} This follows directly from the implication \((5) \Rightarrow (4)\) in Theorem~\ref{thm:main1}.

\textbf{\((2) \Leftrightarrow (3)\):} If \(\phi\) is maxitive, this equivalence follows from \((5) \Leftrightarrow (6)\) in Theorem~\ref{thm:main1}.

\textbf{\((2) \Rightarrow (4)\):} This is immediate.

\textbf{\((4) \Rightarrow (2)\):} We assume that \(\mu\) is atomless and that \(L^2\) is separable. Suppose \(\phi\) is maxitive and law-invariant. Given that \(\phi\) is law-invariant, \(\mu\) is atomless, and \(L^2\) is separable, by \cite[Remark 4.64]{follmer} (a result due to~\cite{jouini}), we know that \(\phi\) is continuous from below. Then, by Theorem~\ref{thm:main1}, we have
\[
\phi(f) = \esssup_X(f - \I) \quad \text{for all } f \in L^\infty.
\]
It suffices to prove that \(\I = 0\) a.s. Suppose, by contradiction, that \(\mu(\I > 0) > 0\). Then,
\[
0 = -\phi(0) = -\esssup_X(-\I) = \essinf_X \I.
\]
Thus, we can find rational numbers \(q_1, q_2 \in \mathbb{Q}\) with \(0 < q_1 < q_2\) such that
\[
\mu(\I > q_2) > 0 \quad \text{and} \quad \mu(\I < q_1) > 0.
\]
Take \(\varepsilon > 0\) such that \(\varepsilon < \mu(\I > q_2)\) and \(\varepsilon < \mu(\I < q_1)\). Since \(\mu\) is atomless, we can find \(A, B \in \XX\) such that \(A \subset \{\I > q_2\}\), \(B \subset \{\I < q_1\}\), and \(\mu(A) = \mu(B) = \varepsilon\). Take \(q \in (q_1, q_2)\). We then have
\[
\phi(q 1_A) = \esssup_X(q 1_A - \I) \leq \esssup_X((q - q_2) 1_A - \I 1_{A^c}) \leq 0.
\]
On the other hand,
\[
\phi(q 1_B) = \esssup_X(q 1_B - \I) \geq \esssup_X((q - q_1) 1_B - \I 1_{A^c}) = q - q_1 > 0.
\]
Since \(\phi\) is law-invariant and \(q 1_A\) and \(q 1_B\) have the same distribution (recall that \(\mu(A) = \mu(B)\)), we have
\[
0 \geq \phi(q 1_A) = \phi(q 1_B) > 0,
\]
which is a contradiction.
\end{proof}

We next prove Theorem~\ref{thm:main3}.

\begin{proof}

\textbf{\((1) \Leftrightarrow (2)\):} It is similar to \textbf{\((1) \Leftrightarrow (2)\)} in Theorem~\ref{thm:main1}.

\textbf{\((3) \Rightarrow (2)\):} It is straightforward.

\textbf{\((2) \Rightarrow (3)\):} 
The restriction \(\phi|_{L^\infty}\) is real-valued and $\sigma$-maxitive on \(L^\infty\). Hence, by Theorem \ref{thm:main1}, there exists a measurable function \(I \colon X \to [0, \infty]\) such that
\[
\phi(f) = {\rm ML}_I(f) \quad \text{for all } f \in L^\infty.
\]
Now, suppose that \(f \in L^0\) is essentially bounded from below. In this case, there exists a sequence \((f_n)_{n \in \mathbb{N}} \subset L^\infty\) such that \(f_n \uparrow f\) a.s. Then,
\[
\phi(f) = \lim_{n \to \infty} \phi(f_n) = \lim_{n \to \infty} {\rm ML}_I(f_n) = {\rm ML}_I(f),
\]
where we have used the fact that \({\rm ML}_I\) is continuous from below.

Finally, consider an arbitrary \(f \in L^0\). Given a fixed \(n \in \mathbb{N}\), since \((-n) \vee f\) is essentially bounded from below, applying the previous step, we obtain
\[
(-n) \vee \phi(f) = \phi\big((-n) \vee f\big) = {\rm ML}_I\big((-n) \vee f\big) = (-n) \vee {\rm ML}_I(f).
\]
Letting \(n \to \infty\), we obtain the desired conclusion.

Finally, note that if \((3)\) holds, then \(I = \I\)~a.s. by Theorem~\ref{thm:main1}.
\end{proof}

We finally prove Theorem~\ref{thm:main4}.

\begin{proof}
Fix any \( B \in \mathcal{X} \). For a given \( k \in \mathbb{N} \), since \(\phi\) is maxitive, we have
\begin{align*}
J_B &\leq J_{(B \cap A_k) \cup A_k^c} \leq J_{B \cap A_k} \vee J_{A_k^c} \leq \left(-\essinf_{B \cap A_k} I\right) \vee J_{A_k^c},
\end{align*}
where the last inequality follows from the condition (3). Letting \( k \to \infty \) and applying the condition (2), we conclude that 
\begin{equation}\label{eq:upperBound}
J_B \leq -\essinf_B I \quad \text{for all } B \in \mathcal{X}.
\end{equation}

Next, define \( A_\infty := \bigcap_{k \in \mathbb{N}} A_k^c \). We claim that
\begin{equation}\label{eq:extended3}
J_B = -\essinf_B I \quad \text{for any } B \subset A_k \text{ and } k \in \mathbb{N} \cup \{\infty\}.
\end{equation}
Since \eqref{eq:extended3} holds for \( k \in \mathbb{N} \) by the condition (3), it suffices to prove the case \( k = \infty \). Indeed, if \( B \subset A_\infty \), then
\[
J_B \leq J_{A_\infty} \leq J_{A_k^c} \quad \text{for all } k \in \mathbb{N},
\]
and letting \( k \to \infty \), the condition (2) yields
\[
J_B = -\infty = -\essinf_B I,
\]
where the last equality uses the condition (1), which implies \( B \subset A_\infty = \{I = \infty\} \).

For each \( k \in \mathbb{N} \cup \{\infty\} \), define \( I_k := I 1_{A_k} + \infty 1_{A_k^c} \). Given \( B \in \mathcal{X} \), using \eqref{eq:extended3}, we have
\[
J_B \geq J_{B \cap A_k} = -\essinf_{B \cap A_k} I = -\essinf_B I_k.
\]
Since \( B \in \mathcal{X} \) was arbitrary, Lemma \ref{lem:IJ} implies \( \I \leq I_k \)~a.s. for all \( k \in \mathbb{N} \cup \{\infty\} \). 

Observe that \( I = I_k \) on \( A_k \) for \( k \in \mathbb{N} \cup \{\infty\} \), and \( X = \bigcup_{k \in \mathbb{N} \cup \{\infty\}} A_k \). Therefore, \( \I \leq I \)~a.s. From \eqref{eq:upperBound} and Lemma \ref{lem:IJ}, we also obtain the reverse inequality, so we conclude that \( \I = I \)~a.s.

Finally, since \( \I = I \)~a.s., it follows from \eqref{eq:relationIJ} in Lemma \ref{lem:IJ} that
\[
J_B \geq -\essinf_B I \quad \text{for all } B \in \mathcal{X}.
\]
Combining this with \eqref{eq:upperBound}, we obtain \( J_B = -\essinf_B I \) for all \( B \in \mathcal{X} \). Therefore, all the conclusions follow from Theorem~\ref{thm:main1}.
\end{proof}

\end{appendix}

\end{document}